\documentclass[11pt]{amsart}
\usepackage{amsmath,mathtools}
\usepackage{amsmath, amsthm, amssymb, amsfonts, enumerate}

\usepackage[colorlinks=true,linkcolor=blue,urlcolor=blue]{hyperref}
\usepackage{dsfont}
\usepackage{color}
\usepackage{geometry}
\usepackage{todonotes}
\usepackage{epstopdf}
\newcommand{\stkout}[1]{\ifmmode\text{\sout{\ensuremath{#1}}}\else\sout{#1}\fi}
\newtheorem{theorem}{Theorem}[section]
\newtheorem{remark}[theorem]{Remark}
\newtheorem{assumption}[theorem]{Assumption}
\newtheorem{lemma}[theorem]{Lemma}
\newtheorem{proposition}[theorem]{Proposition}
\newtheorem{corollary}[theorem]{Corollary}

\def\A{\mathcal{A}}

\def \E{\mathsf{E}}

\def \P{\mathsf{P}}
\def \R{\mathbb{R}}
\def \F{\mathbb{F}}

\def\d{\mathrm{d}}

\newcommand{\esssup}{\text{esssup}}
\newcommand{\essinf}{\text{essinf}}
\newcommand{\diver}{\text{div}}

\definecolor{red}{rgb}{1.0,0.0,0.0}

\definecolor{blu}{rgb}{0.0,0.0,1.0}

\definecolor{gre}{rgb}{0.03,0.50,0.03}

\title[Infinite-Dimensional Singular Stochastic Control]{On a Class of Infinite-Dimensional Singular Stochastic Control Problems}
\author[Federico]{Salvatore Federico}
\author[Ferrari]{Giorgio Ferrari}
\author[Riedel]{Frank Riedel}
\author[R\"ockner]{Michael R\"ockner}

\address{S.~Federico: Dipartimento di Economia Politica e Statistica, Universit\`a di Siena, Piazza san Francesco 7/8, 53100, Siena Italy}
\email{\href{mailto:salvatore.federico@unisi.it}{salvatore.federico@unisi.it}}
\address{G.~Ferrari: Center for Mathematical Economics (IMW), Bielefeld University, Universit\"atsstrasse 25, 33615, Bielefeld, Germany}
\email{\href{mailto:giorgio.ferrari@uni-bielefeld.de}{giorgio.ferrari@uni-bielefeld.de}}
\address{F.~Riedel: Center for Mathematical Economics (IMW), Bielefeld University, Universit\"atsstrasse 25, 33615, Bielefeld, Germany}
\email{\href{mailto:frank.riedel@uni-bielefeld.de}{frank.riedel@uni-bielefeld.de}}
\address{M.~R\"ockner: Faculty of Mathematics, Bielefeld University, Universit\"atsstrasse 25, 33615, Bielefeld, Germany}
\email{\href{mailto:roeckner@math.uni-bielefeld.de}{roeckner@math.uni-bielefeld.de}}

\date{\today}

\numberwithin{equation}{section}

\begin{document}

\begin{abstract} 
We study a class of infinite-dimensional singular stochastic control problems with applications in economic theory and finance. The control process linearly affects an abstract evolution equation on a suitable partially-ordered infinite-dimensional space $X$, it takes values in the positive cone of $X$, and it has right-continuous and nondecreasing paths. We first provide a rigorous formulation of the problem by properly defining the controlled dynamics and integrals with respect to the control process. We then exploit the concave structure of our problem and derive \emph{necessary and sufficient} first-order conditions for optimality. The latter are finally exploited in a specification of the model where we find an explicit expression of the optimal control. The techniques used are those of semigroup theory, vector-valued integration, convex analysis, and general theory of stochastic processes. 

\end{abstract}

\maketitle

\smallskip


\smallskip

{\textbf{MSC2010 subject classification}}: 93E20, 37L55, 49K27, 40J20, 91B72.

%%%%%%%%%%%%%%%%%%%%%%%%%%%%%%%%%%%%%%%%%

\section{Introduction}
\label{introduction}

In this paper we study a class of infinite-dimensional singular stochastic control problems, over the time-interval $[0,T]$, where $T \in (0,\infty]$. As we discuss below, these are motivated by relevant models in economic theory and finance.

Let $(\Omega,\mathcal{F},\mathbb{F}:=(\mathcal{F}_t)_{t\in[0,T]},\P)$ be a filtered probability space, let $(D,\mathcal{M},\mu)$ be a measure space, and let $X:=L^p(D,\mathcal{M},\mu)$, where $p \in (1,\infty)$. The state variable $(Y_t)_{t\in[0,T]}$ of our problem is a stochastic process evolving in the space $X$ according to a linear (random) evolution equation, that is linearly affected by the control process $\nu$:
\begin{equation}
\label{state:intro}
\d Y_t=\A Y_t\,\d t+ \d \nu_t.
\end{equation}
The stochastic process $(\nu_t)_{t\in [0,T]}$ is adapted with respect to the reference filtration $\F$, it has right-continuous and nondecreasing paths, and takes values in the positive cone of $X$. Among other more technical conditions, we assume that the operator $\A$ above generates a $C_0$-semigroup of positivity-preserving bounded linear operators $(e^{t\A})_{t\geq 0}$ in the space $X$. The performance criterion to be maximized takes the form of an expected net profit functional. The randomness comes into the problem through an exogenous $X^*$-valued process $(\Phi^*_t)_{t\in[0,T]}$ -- where $X^*$ is the topological dual of $X$ -- giving the marginal cost of control, and through a general random running profit/utility function $\Pi:\Omega\times [0,T]\times X\to\R_+$. That is, we consider a functional of the form
\begin{equation}
\label{fun:intro}
\mathcal{J}(\mathbf{y},\nu):=\E\bigg[\int_0^T \Pi\left(t, Y_t^{\mathbf{y},\nu}\right) q(\d t) -  \int_0^T \langle \Phi^*_t,\,\d \nu_t\rangle_{X^*\!,\,X} \bigg].
\end{equation}
Here $\mathbf{y}$ denotes the initial state of the system, $q$ is a suitable measure on $[0,T]$, and $\langle \cdot,\cdot\rangle_{X^*,X}$ denotes the dual pairing of the spaces $X^*,X$.
\vspace{0.15cm}

\textbf{Related Literature.}\ In finite-dimensional settings, singular stochastic control problems and their relation to questions of optimal stopping are nowadays a well-established brunch of optimal control theory, that found (and, actually, were motivated by) many applications in several contexts (see, e.g, Chapter VIII in \cite{FlemingSoner}). While the theory of regular stochastic control and of optimal stopping in infinite-dimensional (notably, Hilbert) spaces received a large attention in the last decades (see, e.g., the recent monography \cite{FGS} for control problems, and \cite{BM}, \cite{CDA}, \cite{FO}, \cite{GS} for optimal stopping), the literature on singular stochastic control in infinite-dimensional spaces is very limited. The only two papers brought to our attention are \cite{Oksendal1} and \cite{Oksendal2}, where the authors study problems motivated by optimal harvesting in which a stochastic partial differential equation (SPDE) is controlled through a singular control. In \cite{Oksendal1} the problem is posed for a quite general controlled SPDE, which also enjoys a space-mean dependence in \cite{Oksendal2}. The authors derive a necessary Maximum Principle, which is also sufficient under the assumption that the Hamiltonian function related to the considered control problem is concave. However, these yet valuable contributions seem to suffer from the foundational point of view, since when dealing with (singularly controlled) SPDEs one has to be cautious with existence of a solution and application of It\^o's formula (see \cite{LR} for theory and results on SPDEs). In particular, it turns out that in infinite-dimensional singular (stochastic) control problems already the precise meaning of the integral with respect to the vector measure represented by the control process - and therefore the precise meaning of the controlled state equation - is a delicate issue that deserves to be addressed carefully. 
\vspace{0.15cm}   
 
\textbf{Contribution and Results.}\ The contribution of this work is threefold.

First of all, our work aims at having a foundational value by providing a rigorous framework where to formulate singular (stochastic) control problems in infinite-dimensional spaces. In this respect, it is worth stressing that we have a different view on the controlled state equation with respect to \cite{Oksendal2}, \cite{Oksendal1}: whereas the latter works follow a variational approach, we follow a semigroup approach (see \cite{BDDM}, \cite{DPZ} for comparison in different contexts).
In particular, in order to make the controlled dynamics well defined as a \emph{mild solution} (see \eqref{mild}) to the singularly controlled (random) evolution equation on a suitable space $X$, we need to properly define time-integrals in which the semigroup generated by the operator $\A$ is integrated against the differential (in time) of the control process. Moreover, to perform our study, also integrals of $X^*$-valued stochastic processes with respect to the differential (in time) of the control process have to be introduced, and a related theorem of Fubini-Tonelli's type has to be proved. All those definitions and results are based on the identification of any control process with a (random) countably additive vector measure on the Borel $\sigma$-algebra of $[0,T]$, and on the so-called Dunford-Pettis' theorem (see Section \ref{sec:setting} below for more details). To the best of our knowledge, such a rigorous foundation of the framework appears in this work for the first time, and we believe that this contribution can pave the way to the study of other infinite-dimensional singular stochastic control problems. 

Second of all, by exploiting the linearity of the controlled state-variable with respect to the control process, and the concavity of the profit functional, we are able to derive \emph{necessary and sufficient} first-order conditions for optimality. These can be seen as a generalization, in our stochastic and infinite-dimensional setting, of the Kuhn-Tucker conditions of classical static optimization theory, and they are consistent with those already obtained for finite-dimensional singular stochastic control problems (see \cite{BankRiedel1}, \cite{Bank05}, and \cite{FerrariSalminen}, among others). It is worth noting that for this derivation, the operator $\A$, as well as the random profit function and the marginal cost of control, are quite general.

Clearly, further requirements are needed in order to provide an explicit solution to our problem. The third main contribution of this work is the determination of an explicit expression for the optimal control, in a setting that is more specific, but still general enough. In particular, we assume that $\A$ generates a $C_0$-group of operators, that the unitary vector $\mathbf{1}$ is an eigenvector of $\A$ and of its adjoint $\A^*$, and that the random profit and the random marginal cost of investment are proportional - through real-valued stochastic processes - to such a unitary vector. However, the initial distribution $\mathbf{y}$ of the controlled state variable is an arbitrary vector belonging to the positive cone of $X$, thus still providing an infinite-dimensional nature to the control problem. Under these specifications, we show that, if $\mathbf{y}$ is sufficiently small, then an optimal control is given in terms of the real-valued optional process $(\ell_t)_{t\in[0,T]}$ uniquely solving a one-dimensional backward equation \`a la Bank-El Karoui. The optimal control prescribes making an initial jump of space-dependent size $\mathbf{1}(x)\ell_0 - \mathbf{y}(x)$, $x \in D$. Then, at any positive time, the optimal control keeps the optimally controlled dynamics proportional to the unitary vector, and with a shape that is given by the running supremum of $\ell$. To the best of our knowledge, this is the first paper providing the explicit solution to an infinite-dimensional singular stochastic control problem. Indeed, in Section 3 of \cite{Oksendal2} and Section 2.1 of \cite{Oksendal1} only heuristic discussion on the form of the optimal control is presented.
\vspace{0.15cm}

\textbf{Economic Interpretation and Potential Models.}\ The class of infinite-dimensional singular stochastic control problems that we study in this paper has important potential applications in economics and finance, and we now provide an informal discussion on that.
\vspace{0.15cm}

\emph{Irreversible Investment.}\ Investment in skills, capacity, and technology is often irreversible (see \cite{DP}). Due to the considerable complexity of intertemporal profit maximization problems involving irreversible decisions, most of the literature is confined to single product decisions. Our setup allows to take the full heterogeneity of investment opportunities into account.

For example, think of a globally operating firm that can invest at various geographic locations, in various types of workers with location-specific skills and education levels, with varying natural environments for machines and buildings. Then, the different parameters of investment can be described, for instance, by a parameter $x \in D \subseteq \mathbb R^n$. The firm controls the cumulative investment $\nu_t(x)$ up to time $t$ at each location-skill-environment parameter $x$, resulting in an overall production capacity $Y_t(x)$. Due to demographic changes, changes in the natural environment, or spillover effects, the various capacities evolve locally in space according to an operator $\mathcal A$. The dynamics is therefore given by an evolution equation of type \eqref{state:intro}, with $D \subseteq \R^n$ and $\mu$ the Lebesgue measure. The firm faces stochastic marginal costs of investment, $\Phi^*$, and running profits depending on the current level of production capacity and, possibly, on other stochastic factors affecting the business conditions. The aim is to maximize expected net profits over a certain time horizon $[0,T]$, i.e.\ a functional of type \eqref{fun:intro}. A specific example of such an irreversible investment problem is solved explicitly in Section \ref{sec:explicit} below.
\vspace{0.15cm}

\emph{Monopolistic Competition.}\ The theory of monopolistic competition is a classic in economics that has been proposed in \cite{Ch} as alternative to the Walras--Arrow--Debreu paradigm of competitive markets. It is used frequently in international economics (see, e.g., \cite{KO}).

In monopolistic competition, a large group of firms produces differentiated commodities (``brands''). Each firm has a local monopoly for its own brand. However, there is competition in the sense that customers might well be able to substitute one brand for another, for example if the brands just differ in quality, but not in the essential economic use. Consumer's intertemporal welfare might be described by a constant elasticity of substitution utility functional of the form
$$\int_0^T \left(\int_D Y_t(x)^{1-\gamma} \mu(\d x) \right)^{\frac{1}{1-\gamma}}\, q(\d t), \quad \gamma \in (0,1),$$
with the measure $\mu$ describing the weight or importance of each brand for welfare, and the measure $q$ the time-preferences of the agent. Here, $Y_t(x)$ is aggregate consumption of brand $x$ at time $t$. Its evolution is driven by an operator $\A$ that might take into account any possible interaction across the different firms (spillover effects, technological shifts, etc.), and its level can be instantaneously increased by the agent through consumption. Hence, we might think that $Y$ evolves as in \eqref{state:intro}. 

We may also assume that the consumer faces a linear budget constraint for the \emph{ex ante} price of a consumption plan $\nu$, with stochastic time-varying marginal price of consumption $(\Phi^*_t)_{t \in [0,T]}$, i.e.
$$
\E\left[\int_0^T \int_D \Phi^*_t(x)\d \nu_t(x)\mu(\d x)\right]\leq w,
$$
for some initial wealth $w>0$.
Then, after deriving the Lagrangian functional associated to such an intertemporal optimal consumption problem, one easily realizes that efficient allocations can be found by solving a control problem of the form \eqref{fun:intro}. Our approach thus gives a rigorous foundation for studying monopolistic competition involving multiple commodities and irreversible consumption decisions.
\vspace{0.15cm}

\emph{Intertemporal Consumption with Substitution and Commodity Differentiation.}\ The lifecycle consumption choice model forms a basic building block for most macroeconomic and financial market models (cf., e.g., \cite{Cochrane}). So far, most intertemporal consumption models suppose a single consumption good and assume a time-additive expected utility specification in order to keep the mathematics simple and to allow for explicit solutions. One thus thinks of the consumption good as an aggregate commodity which reflects the overall consumption bundle. Consumption occurs, however, in many different goods and quality levels. Moreover, the time-additive structure of utility functions ignores important aspects of intertemporal substitution as Hindy, Huang, and Kreps have pointed out in \cite{HHK}. 
 
Our work provides a basis to study Hindy--Huang--Kreps utility functionals for differentiated commodities. Consider an agent who can choose at time $t$ consumption from a whole variety of goods $x \in D$, where $D \subseteq \mathbb R^n$. Let $Y_t(x)$ describe the level of satisfaction derived up to time $t$ of variety $x$. The natural evolution of satisfaction along the variety space might be described by a partial differential operator $\A$, which includes depreciation and other changes. The agent increases her level of satisfaction by consuming, and the cumulative consumption of variety $x$ is described by $\nu_t(x)$, which is an adapted stochastic process, nondecreasing in $t$. The overall level of satisfaction then evolves through a controlled evolution equation like ours \eqref{state:intro} above.
 
Within this setting, the natural extension of the Hindy--Huang--Kreps utility functional takes the form 
$$\E\bigg[\int_0^T \Big(\int_D u(t, Y^{\nu}_t(x)) \mu(\d x) \Big) \d t\bigg],$$
for some measure $\mu$ on $D$, and a (possibly random) instantaneous utility function $u$. Then, if the agent faces a linear budget constraint for the \emph{ex ante} price of a consumption plan $\nu$, with stochastic marginal price of consumption $(\Phi^*_t)_{t \in [0,T]}$ as in the example above, the Lagrangian formulation of the resulting optimal consumption problem leads to an optimal control problem like ours \eqref{fun:intro} (see \cite{BankRiedel1} for a related problem and approach in a finite-dimensional setting).
\vspace{0.15cm}

\textbf{Organization of the Paper.}\ The rest of the paper is organized as follows. In Section \ref{sec:setting-OC} we introduce the setting and formulate the infinite-dimensional singular stochastic control problem. In Section \ref{Solution} we characterize optimal controls via necessary and sufficient conditions. These are then employed in Section \ref{sec:group} to construct an explicit solution in the case when the operator $\A$ generates a $C_0$-group of operators (see in particular Section \ref{sec:explicit}). Applications to PDE models are then discussed in Section \ref{Sec:PDE}, while concluding remarks and future outlooks are presented in Section \ref{sec:conclusion}.

%%%%%%%%%%%%%%%%%%%%%%%%%%%%%%%%%%%%%%%%%%%%%%%%%%%%%%%%%%%%%%%%%%%%%%%%%%%%%%%%%%%%%%%%%%%%%%%%%%%%%%%%%%%%%%%%%%%%%%%%%%%%%%%%%%%%%%%%%%%%%%%%%%%%%%%%%%%%%%%%%%%%%%%%%%

\section{Setting and Problem Formulation}
\label{sec:setting-OC}

\subsection{Setting and Preliminaries}
\label{sec:setting}

Let $(D,\mathcal{M},\mu)$ be a measure space and consider the reflexive Banach space $X:=L^p(D,\mathcal{M},\mu)$, $p\in (1,\infty)$. We denote the norm of $X$ by $|\cdot|_X$. Let $p^*=\frac{p}{p-1}\in (1,\infty)$ be the conjugate exponent of $p$, so that  $X^*:=L^{p^*}(D,\mathcal{M},\mu)$ is the topological dual of $X$. The norm of $X^*$ will be denoted by $|\cdot|_{X^*}$ and the duality pairing between $v^*\in X^*$ and $v\in X$ by $\langle v^*, v \rangle$. The order relations, as well as the supremum or infimum of elements of $X$ and $X^*$, will be intended pointwise. The nonnegative cones of $X$ and $X^*$ are defined, respectively, as 
$$K_+:=\big\{v\in X: \ v \geq 0\big\}, \ \ \ \ K^*_+:=\big\{v^*\in X^*: \ v^* \geq 0\big\}.$$
Hereafter, we denote by $\mathcal{L}(X)$ the space of linear bounded operators $P:X\to X$ and by {$\mathcal{L}^+(X)$} the subspace of  positivity-preserving operators of $\mathcal{L}(X)$; i.e., ${P}\in \mathcal{L}^+(X)$ if
{$$
f\in X, \ f \geq 0  \ \Longrightarrow \  f\geq 0.$$
} 
Throughout the paper, we consider a linear operator $\mathcal{A}:\mathcal{D}(\mathcal{A})\subseteq X\to X$ satisfying the following
\begin{assumption}
\label{ass:A}
$\mathcal{A}$ is closed, densely defined, and generates a $C_0$-semigroup of linear operators $(e^{t\A})_{t\geq 0}\subseteq \mathcal{L}^+(X)$.
\end{assumption}

For examples see Remark \ref{rem:operators} and Section \ref{Sec:PDE} below. Recall that, by classical theory of $C_0$-semigroups (see \cite{EN}), also the adjoint operator $\mathcal{A}^*: \mathcal{D}(\A^*)\subseteq X^*\to X^*$ generates a $C_0$-semigroup on $X^*$; precisely, $e^{t\A^*}=(e^{t\A})^*$. It is easily seen that also $e^{t\mathcal{A}^*} \in \mathcal{L}^+(X^*).$
\medskip

Let $T \leq \infty$ be a fixed horizon\footnote{When $T=\infty$, we shall use the convention that the intervals $[s,T]$ and $(s,T]$, with $s\geq 0$, denote $[s,\infty)$ and $(s,\infty)$, respectively.} and endow the interval $[0,T]$ with the Borel $\sigma$-algebra $\mathcal{B}([0,T])$. 
Also, let $(\Omega,\mathcal{F},\F,\P)$ be a filtered probability space, with filtration $\F:=(\mathcal{F}_t)_{t\in[0,T]}$ satisfying the usual conditions. In the following, all the relationships involving $\omega\in\Omega$ as hidden random parameter are intended to hold $\P$-almost surely. Also, in order to simplify the exposition, often we will not stress the explicit dependence of the involved random variables and processes with respect to $\omega\in \Omega$. Let 
\begin{eqnarray}
\label{set:S}
\mathcal{S}&:=& \{\nu: \Omega \times [0,T] \to K_+:\,\,\F-\text{adapted and such that}\,\, t \mapsto \nu_t \nonumber \\
&& \hspace{1.5cm} \text{is \ nondecreasing and right-continuous}\}.
\end{eqnarray}
In the following, we set $\nu_{0^-}:=\textbf{0}\in K_+$ for any $\nu \in \mathcal{S}$ (see Remark \ref{rem:nu0} below).

Notice that any given $\nu \in \mathcal{S}$ can be seen as a (random) countably additive vector measure $\nu: \mathcal{B}([0,T]) \to K_+$ of finite variation
defined as  
$$\nu([s,t]):= \nu_t - \nu_{s^-} \  \ \ \forall s,t\in [0,T], \ s\leq t.$$ 
We denote by $|\nu|$ the variation of $\nu$, which is a nonnegative (optional random) measure on $([0,T],\mathcal{B}([0,T])$ and that, due to monotonicity of $\nu$, can be simply expressed as 
$$|\nu|([s,t]) = |\nu_t - \nu_{s^-}|_X, \ \ \ \forall s,t\in [0,T], \ s\leq t.$$

\begin{remark}
\label{rem:nu0}
By setting $\nu_{0^-}:=\textbf{0}\in K_+$ for any $\nu \in \mathcal{S}$, we mean that we extend any $\nu \in \mathcal{S}$ by setting $\nu \equiv \textbf{0}$ on $[-\varepsilon,0)$, for a given and fixed $\varepsilon>0$. In this way, the associated measures an have a positive mass at initial time of size $\nu_0$. Notice that this is equivalent with identifying any control $\nu$ with a countably additive vector measure $\nu: \mathcal{B}([0,T]) \to K_+$ of finite variation defined as $\nu((s,t]):= \nu_t - \nu_{s}$, for every $s,t\in [0,T]$, $s < t$, plus a Dirac-delta at time $0$ of amplitude $\nu_0$.
\end{remark}

Since $X$ is a reflexive Banach space, by \cite{DU}, Corollary 13 at p.\ 76 (see also Definition 3 at p.\ 61), there exists a Bochner measurable function  $\rho=\rho(\omega):[0,T] \mapsto K_+$, for a.e.\ $\omega\in\Omega$, such that 
\begin{equation}
\label{DP}
\int_{[0,T]} |\rho_t|_{X} \d|\nu|_t < \infty \quad \text{and} \quad \d\nu= \rho\, \d|\nu|.
\end{equation}
Notice that, seen as a stochastic process, $\rho=(\rho_t)_{t\in [0,T]}$ is $\mathbb{F}$-adapted, because so is $\nu$.
Then, for a given $K^*_+-$valued $\mathbb{F}$-adapted process ${f^*}:=(f^*_t)_{t\in[0,T]}$, in view of \eqref{DP}, for any $t\in [0,T]$ we define 
\begin{eqnarray}
\label{def-int1}
& \displaystyle \int_0^t \langle f^*_s,\,\d \nu_s\rangle :=\int_{[0,t]} \langle f^*_s, \rho_s\rangle \d |\nu|_s = \int_{[0,t]} \Big(\int_D f^*_s(x) \rho_s(x) \d x \Big) \d |\nu|_s \nonumber \\
&= \displaystyle \int_D \Big(\int_{[0,t]} f^*_s(x) \rho_s(x) \d |\nu|_s \Big) \mu(\d x),
\end{eqnarray}
where the last step is possible due to Fubini-Tonelli's theorem. As a byproduct of \eqref{def-int1}, of Fubini-Tonelli's theorem, and of Theorem 57 at p.\ 122 \cite{DM}, Chapter VI, we also have that for any $t\in [0,T]$ 
\begin{equation}
\label{DM-opt}
\E\bigg[\int_0^t \langle \psi^*_s,\,\d \nu_s\rangle\bigg] = \E\bigg[\int_0^t \langle \E\big[\psi^*_s\,|\,\mathcal{F}_s\big],\,\d \nu_s\rangle\bigg],
\end{equation}
for any measurable $K_+$-valued stochastic process $\psi^*=(\psi^*_t)_{t\in[0,T]}$.

Next, given a strongly continuous map $\Theta:[0,T]^2\to \mathcal{L}^+(X)$, $(u,r)\mapsto \Theta(u,r)$, i.e.\ such that $(u,r)\mapsto \Theta(u,r)\mathbf{y}$ is continuous for each $\mathbf{y}\in X$,  we define
\begin{equation}
\label{def-int2}
\int_{0}^t \Theta(t,s) \d \nu_s:= \int_{[0,t]} \Theta(t,s)\rho_s \d |\nu|_s, \ \ \ \ t\in[0,T].
\end{equation}
where the last $X$-valued integral is well defined pathwise in $\Omega$ in the Bochner sense. Indeed, on the one hand, strongly continuity of $\Theta$ and Bochner measurability of $s\mapsto \rho_s(\omega)$ yields Bochner measurability of  $s\mapsto \Theta(t,s)\rho_s(\omega)$. On the other hand, 
 by strong continuity, the set  $\{\Theta(t,s)x, \, s \in [0,t]\}$ is compact in $X$, hence bounded; so, by the uniform boundedness principle, we have 
$$\int_{[0,t]} |\Theta(t,s)\rho_s|_X \d |\nu|_s \leq c \int_{[0,T]} |\rho_s|_X \d |\nu|_s < \infty,$$ 
where $c:= c(t)= \sup_{s\in[0,t]}|\Theta(t,s)|_{\mathcal{L}(X)} < \infty$ and \eqref{DP} has been used. 
\smallskip

The following Tonelli's type result is needed in the next section. 
\begin{lemma}
\label{lemma:fubini}
Let $f^*:[0,T]\times \Omega \to K^*_+$ be a measurable process, let $\Theta:[0,T]^2\to \mathcal{L}^+(X)$ be strongly continuous, and let $q$ be a finite (nonnegative) measure on $([0,T],\mathcal{B}([0,T]))$. Then for any $\nu \in \mathcal{S}$ we have
$$
\int_0^T  \left\langle f^*_t, \int_0^t \Theta(t,s) \d \nu_s \right\rangle q(\d t)= \int_0^T \left\langle   \int_s^T \Theta^*(t,s) f^*_t\, q(\d t), \d \nu_s\right\rangle.
$$
\end{lemma}
\begin{proof}
Let $\nu \in \mathcal{S}$ and fix $\omega\in\Omega$ -- a random parameter that will not be stressed as an argument throughout this proof. Recall that $\d|\nu|$ denotes the finite-variation measure on $[0,T]$ associated to the $K_+$-valued finite-variation measure  $d\nu$ on $[0,T]$. Then, by \eqref{def-int2} and classical Tonelli's theorem 
\begin{eqnarray*}
& \displaystyle \int_0^T  \left\langle f^*_t, \int_0^t \Theta(t,s) \d \nu_s \right\rangle q(\d t) = \int_0^T  \left\langle f^*_t, \int_0^t \Theta(t,s) \rho_s \d|\nu|_s \right \rangle \, q(\d t) \nonumber \\
& \displaystyle = \int_0^T  \int_0^T \mathds{1}_{\{(\bar{s},\bar{t}) \in [0,T]^2:\,\bar{s}\leq \bar{t}\}}(t,s) \big\langle f_t^*,\Theta(t,s) \rho_s \big\rangle \d|\nu|_s \, q(\d t) \nonumber \\
& \displaystyle = \int_0^T  \int_s^T \big\langle f^*_t,\Theta(t,s) \rho_s \big \rangle q(\d t)  \, \d|\nu|_s 
\displaystyle = \int_0^T  \int_s^T \big\langle \Theta^*(t,s)f^*_t q(\d t) , \rho_s \big \rangle  \, \d|\nu|_s\\&
\displaystyle\int_0^T  \left\langle\int_s^T  \Theta^*(t,s) f^*_t\, q(\d t),  \, \rho_s  \d|\nu|_s \right \rangle=  \int_0^T \left\langle \int_s^T \Theta^*(t,s) f^*_t\, q(\d t), \d \nu_s \right \rangle, \nonumber 
\end{eqnarray*}
concluding the proof.
\end{proof}

%%%%%%%%%%%%%%%%%%%%%%%%%%%%%%%%%%%%%%%%%%%%

\subsection{The Optimal Control Problem}
\label{sec:OCproblem}

Bearing in mind the definitions of the last section, for any given and fixed $\mathbf{y}\in K_+$ and $\nu \in \mathcal{S}$, we now consider the abstract equation in $X$:
\begin{equation}
\label{eq:abstract}
\begin{cases}
\d Y_t=\A Y_t\,\d t+ \d \nu_t,\\
Y_{0^-}=\mathbf{y}.
\end{cases}
\end{equation}
By writing $Y_{0^-}=\mathbf{y}$ we intend that we set $Y \equiv \mathbf{y}$ on $[-\varepsilon,0)$, for a given and fixed $\varepsilon>0$. In this way, $Y$ might have an initial jump of size $Y_0 - \mathbf{y}$, due to a possible initial jump of the right-continuous process $\nu$ (cf.\ Remark \ref{rem:nu0}).
Following the classical semigroup approach (see, e.g., \cite{EN}), for any $t\geq 0$, we define the \emph{mild solution} to \eqref{eq:abstract} to be the process
\begin{equation}
\label{mild}
Y_t^{\mathbf{y},\nu}:=e^{t\A} \mathbf{y}+\int_0^t e^{(t-s)\A}\, \d \nu_s, \qquad Y_{0^-}=\mathbf{y}.
\end{equation} 
The expression above can be thought of as the counterpart, in an abstract setting, of the so-called ``variation of constants formula'' of the finite-dimensional setting, and it allows to give a rigorous sense to \eqref{eq:abstract} even when the initial datum $\mathbf{y} \notin \mathcal{D}(\mathcal{A})$. Notice that since $(e^{t\A})_{t\geq 0}$ is positivity-preserving, $y\in K_+$, and $\nu$ is nondecreasing, we have that $Y^{\mathbf{y},\nu}:=(Y_t^{\mathbf{y},\nu})_{t\geq 0}$ takes values in $K_+$. 

Let $\Phi^*$ be an $\mathbb{F}$-adapted $K^*_+$-valued stochastic process with c\`adl\`ag (right-continuous with left-limits) paths, and take $\Pi:\Omega \times [0,T]\times K_+\to \R_+$ measurable. We define the convex set of \emph{admissible controls}
\begin{equation}
\label{set:C}
\mathcal{C}:=\Big\{\nu \in \mathcal{S}:\,\,\E\left[\int_0^T\langle \Phi^*_t,\,\d \nu_t\rangle \right] < \infty\Big\}.
\end{equation}
Then, for any $\mathbf{y} \in K_+$, $\nu\in\mathcal{C}$ 
%\begin{equation}
%\label{eq:finitecost}
%\E\left[\int_0^T\langle \Phi^*_t,\,\d \nu_t\rangle \right] < \infty,
%\end{equation}
we consider the performance criterion
$$\mathcal{J}(\mathbf{y},\nu):=\E\left[\int_0^T \Pi\left(t, Y_t^{\mathbf{y},\nu}\right) q(\d t) -  \int_0^T \langle \Phi^*_t,\,\d \nu_t\rangle \right],$$
where $q$ is a finite nonnegative measure on $([0,T],\mathcal{B}([0,T]))$. Note that $\mathcal{J}(\mathbf{y},\nu)$ is well defined (possibly equal to $+\infty$), due to the definition of $\mathcal{C}$.

We then consider the following optimal control problem: 
\begin{equation}
\label{OCproblem}
v(\mathbf{y}):=\sup_{\nu\in\mathcal{C}} \mathcal{J}(\mathbf{y},\nu).
\end{equation}
Clearly, denoting by $\mathbf{0}$ the null element of $\mathcal{C}$, we have $\mathcal{J}(\mathbf{y},\mathbf{0}) \geq0$. Hence,
$$
0\leq v(\mathbf{y}) \leq +\infty \ \ \ \forall \mathbf{y}\in K_+.
$$
We say that $\nu^{\star}\in\mathcal{C}$ is optimal for problem \eqref{OCproblem} if it is such that $\mathcal{J}(\mathbf{y},\nu^{\star})=v(\mathbf{y})$.

\begin{remark} 
In the following results of this paper, the choice of $X=L^p(D,\mathcal{M},\mu)$ with $p\in(1,\infty)$ is not strictly necessary. Indeed, what we really use is that $X$ is a reflexive Banach lattice. 
\end{remark}

%%%%%%%%%%%%%%%%%%%%%%%%%%%%%%%%%%%%%%%%%%%%%%%%%%%%%%%%%%

\section{Characterization of optimal controls by necessary and sufficient first-order conditions}
\label{Solution}

In this section we derive sufficient and necessary conditions for the optimality of a control $\nu^{\star}\in\mathcal{C}$. Let us introduce the set 
$$
S_{\mathbf{y}}(t):=\{\mathbf{k}\in K_+: \ \mathbf{k}\geq e^{t\A}\mathbf{y}\}, \ \ \ t\in[0,T].
$$
Notice that the positivity-preserving property of the semigroup $e^{t\mathcal{A}}$ yields 
$$
 Y^{\mathbf{y},\nu}_t\in S_{\mathbf{y}}(t), \ \  \forall t\in [0,T], \ \forall \nu\in \mathcal{C}.
$$
The next assumption will be standing thoroughout the rest of this paper.
\newpage 

\begin{assumption}
\label{ass:U}
\begin{enumerate}[(i)]
\item[]
\item $\Pi:\Omega\times[0,T]\times K_+\to\R_+$ is such that 
$\Pi(\omega, t,\cdot)$ is concave, nondecreasing, and of class $C^1(S_{\mathbf{y}}(t);\R)$ for each $(\omega,t)\in\Omega\times [0,T]$. Moreover, for any $z \in K_+$, the stochastic process $\Pi(\cdot,\cdot,z): \Omega \times [0,T] \to \R_+$ is $\mathbb{F}$-progressively measurable.
\item $\mathcal{J}(\mathbf{y},\nu)<\infty$  for each $\nu\in\mathcal{C}$.
\end{enumerate}
\end{assumption}

\begin{remark}
\begin{itemize}
\item[(a)] The condition $\mathcal{J}(\mathbf{y},\nu)<\infty$ for each $\nu\in\mathcal{C}$ required in Assumption \ref{ass:U}-(ii) is clearly verified when $\Pi$ is bounded. On the other hand, sufficient conditions guaranteeing Assumption \ref{ass:U}-(ii) in the case of a possibly unbounded $\Pi$ should be determined on a case by case basis as they may depend on the structures of $\Pi$, $\mathcal{A}$, and $\Phi^*$. We will provide a set of such conditions for the separable case studied in Section \ref{sec:explicit} (see Lemma \ref{lem:finitefunctional}).
\item [(b)] Notice that the smoothness condition on $\Pi(\omega,t,\cdot)$ can be relaxed by employing in the following proofs the supergradient of $\Pi$ instead of its gradient. However, we prefer to work under this reasonable regularity requirement in order to simplify exposition.
\item[(c)] If it exists, an optimal control for problem \eqref{OCproblem} is unique whenever $\Pi(\omega, t,\cdot)$ is strictly concave for each $(\omega,t)\in\Omega\times [0,T]$.
\end{itemize}
\end{remark}

In the following, by $\nabla \Pi$ we denote the gradient of $\Pi$ with respect to the last argument. Note that the map
$\nabla \Pi(t,\cdot)$ defined on $\mbox{ri}(K_+)$ takes values in $K^*_+$ by monotonicity of $\Pi(t,\cdot)$, and that it is nonincreasing by concavity of $\Pi(t,\cdot)$.
The following lemma ensures that some integrals with respect to $\d \nu-\d \nu'$ for $\nu,\nu'\in \mathcal{C}$ appearing in our subsequent analysis are well-posed.

\begin{lemma}
\label{lemmafinito} 
Let $\nu\in \mathcal{C}$. Then 
$$
\E\bigg[\int_0^T \Big\langle \int_s^Te^{(t-s)\A^*} \nabla \Pi\left(t,Y_t^{\mathbf{y},\nu}\right) q(\d t), \ \d \nu_s  \Big \rangle\,\d t \bigg]<\infty.
$$
\end{lemma}

\begin{proof} 
Recall that $\mathbf{0}$ denotes the null element of $\mathcal{C}$. Then, using  concavity of $\Pi(t,\cdot)$, \eqref{mild}, Lemma \ref{lemma:fubini} (with $\Theta(t,s)=e^{(t-s)\mathcal{A}}$ and $f^*_t=\nabla \Pi\left(t,Y_t^{\mathbf{y},\nu}\right)$), the fact that $\mathcal{J}(\mathbf{y},\nu)<\infty$ for any $\nu\in \mathcal{C}$ by Assumption \ref{ass:U}-(ii), as well as that $\mathcal{J}(\mathbf{y},\mathbf{0})\geq0$, we can write 
\begin{equation}
\label{eq:proofFOCs}
\begin{split}
\infty> \  \mathcal{J}(\mathbf{y},\nu) - \mathcal{J}(\mathbf{y},\mathbf{0})&=\E\left[\int_0^T\left(\Pi(t,Y_t^{\mathbf{y},\nu})-\Pi(t,e^{t\A}\mathbf{y})\right) q(\d t) - \int_0^T \langle\Phi^*_t,\d \nu_t\rangle\right] \\&\geq 
 \E\bigg[\int_0^T\Big\langle \nabla \Pi\left(t,Y_t^{\mathbf{y},\nu}\right), \ Y_t^{\mathbf{y},\nu} - e^{t\A}\mathbf{y}\Big\rangle \,q(\d t) - \int_0^T \langle\Phi^*_t,\d \nu_t\rangle\bigg]\\
&= \E\bigg[\int_0^T\Big\langle \nabla \Pi\left(t,Y_t^{\mathbf{y},\nu}\right),  \ \int_0^t e^{(t-s)\A} \d \nu_s  \Big\rangle \,q(\d t) - \int_0^T \langle\Phi^*_t,\d \nu_t\rangle \bigg] \\
 &= \E\bigg[\int_0^T \Big\langle \int_s^Te^{(t-s)\A^*} \nabla \Pi\left(t,Y_t^{\mathbf{y},\nu}\right) q(\d t), \ \d \nu_s  \Big \rangle - \int_0^T \langle\Phi^*_s,\d \nu_s\rangle\bigg].
\end{split}  
\end{equation}
The claim follows by definition of $\mathcal{C}$.
\end{proof}

%The next technical result will be used in the proof of Theorem \ref{prop:FOC}.
%
%\begin{lemma}
%\label{sal}
%Let $\mathcal{G}\subseteq \mathcal{F}$ be a $\sigma$-algebra and let $\{Z(x)\}_{x\in D}$ be a family of $\mathcal{G}$-random variables. The map $Q:\Omega\to X$ whose version (in $D$) is defined by 
%$$[Q(\omega)](x):=Z(x)(\omega), \ \ \ x\in D,$$
%is a $\mathcal{G}$-random variable.
%\end{lemma}
%
%\begin{proof}
%{\color{red}{TO BE DONE.}}
%\end{proof}

\begin{theorem}
\label{prop:FOC} 
A control $\nu^{\star}\in\mathcal{C}$ is optimal for Problem \eqref{OCproblem} {if and only if} the following First-Order Conditions (FOCs) hold true:
\begin{enumerate}[(i)]
\item For every $\nu\in \mathcal{C}$
$$\E\left[\int_0^T\Big\langle \E\left[\int_s^Te^{(t-s)\A^*} \nabla \Pi\left(t,Y_t^{\mathbf{y},\nu^{\star}}\right)q(\d t)\,\Big|\,  \mathcal{F}_s\right]-\Phi^*_s,\d \nu_s\Big\rangle  \right]\leq 0;$$
\item the following equality holds:
$$\E\left[\int_0^T\Big\langle \E\left[\int_s^Te^{(t-s)\A^*} \nabla \Pi\left(t,Y_t^{\mathbf{y},\nu^{\star}}\right)q(\d t)\,\Big|\,  \mathcal{F}_s\right]-\Phi^*_s,\d \nu^{\star}_s\Big\rangle  \right]= 0.$$
\end{enumerate}
\end{theorem}

\begin{proof}
\emph{Sufficiency.} Let $\nu^{\star} \in \mathcal{C}$ satisfying (i)--(ii) above, and let $\nu \in \mathcal{C}$ be arbitrary. 
By Lemma \ref{lemma:fubini} we have (after taking expectations)
\begin{equation}\label{Fubini}
\begin{split}
& \E\bigg[\int_0^T\Big\langle \nabla \Pi\left(t,Y_t^{\mathbf{y},\nu^{\star}}\right), \ \int_0^t e^{(t-s)\A} \big(\d \nu^{\star}_s - \d \nu_s\big) \Big\rangle q(\d t)- \int_0^T \langle \Phi^*_s, \d \nu^{\star}_s-\d \nu_s\rangle\bigg] \\
& =  \E\bigg[\int_0^T \Big\langle \int_s^Te^{(t-s)\A^*} \nabla \Pi\left(t,Y_t^{\mathbf{y},\nu^{\star}}\right)q(\d t) -\Phi^*_s, \ \d \nu^{\star}_s - \d \nu_s \Big \rangle\bigg].  
 \end{split}
\end{equation}
Notice that the previous quantity is well defined due to Lemma \ref{lemmafinito}.
Moreover, by \eqref{DM-opt},
\begin{equation}
\label{op}
\begin{split}
& \E\bigg[\int_0^T \Big\langle \int_s^Te^{(t-s)\A^*} \nabla \Pi\left(t,Y_t^{\mathbf{y},\nu^{\star}}\right)q(\d t) -\Phi^*_s, \ \d \nu^{\star}_s - \d \nu_s \Big \rangle\bigg] \\
%&= \E\bigg[\E\left[\int_0^T \Big\langle \int_s^Te^{(t-s)\A^*} \nabla \Pi\left(t,Y_t^{\mathbf{y},\nu^{\star}}\right)q(\d t) -\Phi^*_s, \ \d \nu^{\star}_s - \d \nu_s \Big \rangle\,\Big|\,\mathcal{F}_s\right]\bigg]\\ 
%&= \E\bigg[\int_0^T \E\left[\Big\langle \int_s^Te^{(t-s)\A^*} \nabla \Pi\left(t,Y_t^{\mathbf{y},\nu^{\star}}\right)q(\d t) -\Phi^*_s, \ \d \nu^{\star}_s - \d \nu_s \Big \rangle\,\Big|\,\mathcal{F}_s\right]\bigg]\\ 
& = \E\bigg[\int_0^T \Big\langle \E\bigg[\int_s^Te^{(t-s)\A^*} \nabla \Pi\left(t,Y_t^{\mathbf{y},\nu^{\star}}\right)q(\d t) -\Phi^*_s\,\Big|\, \mathcal{F}_s\bigg],\ \d \nu^{\star}_s - \d \nu_s \Big \rangle \bigg]. 
\end{split}
\end{equation}

Then, using \eqref{Fubini}-\eqref{op}, concavity of $\Pi(t,\cdot)$, \eqref{mild}, and (i)--(ii), we can write 
\begin{equation*}
\label{eq:proofFOCs-bis}
\begin{split}
 &\mathcal{J}(\mathbf{y},\nu^{\star}) - \mathcal{J}(\mathbf{y},\nu) \geq 
 \E\bigg[\int_0^T\Big\langle \nabla \Pi\left(t,Y_t^{\mathbf{y},\nu^{\star}}\right), \ Y_t^{\mathbf{y},\nu^{\star}}-Y_t^{\mathbf{y},\nu} \Big\rangle\,q(\d t) - \int_0^T \langle \Phi^*_t, \d \nu^{\star}_t-\d \nu_t\rangle \bigg]\\
= &\E\bigg[\int_0^T\Big\langle \nabla \Pi\left(t,Y_t^{\mathbf{y},\nu^{\star}}\right),  \ \int_0^t e^{(t-s)\A} \big(\d \nu^{\star}_s - \d \nu_s\big) \Big\rangle \,q(\d t) - \int_0^T\langle \Phi^*_s, \d \nu^{\star}_s-\d \nu_s\rangle \bigg] \\
 = &\E\bigg[\int_0^T \Big\langle \E\bigg[\int_s^Te^{(t-s)\A^*} \nabla \Pi\left(t,Y_t^{\mathbf{y},\nu^{\star}}\right)q(\d t)\,\Big|\, \mathcal{F}_s\bigg] - \Phi^*_s, \ \d \nu^{\star}_s - \d \nu_s \Big \rangle  \bigg] \geq 0.
 \end{split}  
\end{equation*}
The optimality of $\nu^{\star}$ follows.
\smallskip

\emph{Necessity.}  The proof of the necessity of (i) and (ii) requires some more work with respect to that of their sufficiency, and it is organized in three steps.

Let $\nu^{\star}\in\mathcal{C}$ be optimal for Problem \eqref{OCproblem}.
\vspace{0.25cm}

\emph{Step 1.} In this step, we show that $\nu^{\star}$ solves the linear problem
$$
\sup_{\nu\in \mathcal{C}}\E\left[\int_0^T\left\langle\Psi^\star_t,\d \nu_t\right\rangle \right],
$$
where we have set
\begin{equation}
\label{Psistar}
\Psi^\star_t:=\E\left[\int_t^T\nabla \Pi(s,Y_s^{\mathbf{y},\nu^{\star}}) q(\d s)\,\Big|\,\mathcal{F}_t\right] -\Phi^*_t.
\end{equation}
Notice that $\Psi^\star$ is $\mathbb{F}$-adapted. Moreover, it is c\`adl\`ag since, by assumption, so is $\Phi^*$ and the underlying filtration $\mathbb{F}$ is right-continuous.

Let $\nu\in \mathcal{C}$ be arbitrary and set $\nu^\varepsilon:=\varepsilon \nu+(1-\varepsilon)\nu^{\star}$ for $\varepsilon\in(0,1/2]$. Clearly $\nu^\varepsilon\in \mathcal{C}$ by convexity of $\mathcal{C}$. Set $Y:=Y_t^{\mathbf{y},\nu}$, $Y^\varepsilon:= Y^{\mathbf{y},\nu^\varepsilon}$, $Y^\star:=Y^{\mathbf{y},\nu^{\star}}$, and note that, by linearity of \eqref{eq:abstract}, one has $Y^\varepsilon=Y^\star+\varepsilon(Y-Y^\star)$. By concavity of $\Pi$, optimality of $\nu^{\star}$, and Lemma \ref{lemma:fubini}, one can write
\begin{eqnarray}
\label{estimate-nec-1}
0&\geq & \frac{\mathcal{J}(\mathbf{y},\nu^\varepsilon)-\mathcal{J}(\mathbf{y},\nu^{\star})}{\varepsilon}\ =  
\frac{1}{\varepsilon}\,\E\left[\int_0^T\big(\Pi(t,Y_t^\varepsilon)- \Pi(t,Y_t^\star)\big)\d t - \int_0^T \langle \Phi^*_t,\d \nu^\varepsilon_t-\d \nu^{\star}_t\rangle\right] \nonumber\\
%\ 
%\frac{1}{\varepsilon}\,\E\left[\int_0^T\left(\Pi(t,Y_t^*+\varepsilon(Y_t-Y_t^*))-\Pi(t,Y_t^*)\right)\d t\right]\\
&\geq & \frac{1}{\varepsilon}\E\left[\int_0^T\langle \nabla \Pi(t, Y^\varepsilon_t), Y^\varepsilon_t-Y^\star_t\rangle q(\d t)- \int_0^T\langle \Phi^*_t,\d \nu^\varepsilon_t-\d \nu^{\star}_t\rangle \right] \nonumber \\
&=& \E\left[\int_0^T\langle \nabla \Pi(t, Y^\varepsilon_t), Y_t-Y^\star_t\rangle q(\d t) - \int_0^T \langle \Phi^*_t,\d \nu_t-\d \nu^{\star}_t \rangle\right]  \\
&=& \E\left[\int_0^T \left \langle \nabla \Pi(t,Y_t^\varepsilon), \ \int_0^t e^{(t-s)\mathcal{A}}(\d \nu_s-\d \nu^{\star}_s) \right \rangle q(\d t) - \int_0^T \langle \Phi^*_s,\d \nu_s-\d \nu^{\star}_s\rangle\right] \nonumber \\
& = & \E\left[\int_0^T\left \langle \int_s^Te^{(t-s)\mathcal{A}^*}\nabla \Pi(t,Y^\varepsilon_t)q(\d t) -\Phi^*_s,\, \d \nu_s- \d \nu^{\star}_s\right\rangle \right]. \nonumber
\end{eqnarray}
We notice that the last expectation above is well-defined. Indeed, observing that $Y^{\varepsilon} \geq \frac{1}{2} Y^{\star}$ and that $\nabla \Pi(t,\cdot)$ is nondecreasing, we can write
\begin{align}
\label{integrabilityepsi-1}
& - \int_0^T \langle \Phi^*_s,\,\d \nu^{\star}_s \rangle \  \leq \ \int_0^T\left\langle \int_s^Te^{(t-s)\mathcal{A}^*}\nabla \Pi(t,Y^\varepsilon_t)q(\d t) -\Phi^*_s,\,\d \nu^{\star}_s\right\rangle \nonumber \\
& \leq 2\, \int_0^T\left\langle \int_s^Te^{(t-s)\mathcal{A}^*}\nabla \Pi\Big(t,\frac{1}{2}Y^{\star}_t\Big)q(\d t) -\Phi^*_s,\,\d \Big(\frac{1}{2}\nu^{\star}_s\Big)\right\rangle \\
& = 2\, \int_0^T\left\langle \int_s^Te^{(t-s)\mathcal{A}^*}\nabla \Pi\Big(t,Y^{\frac{1}{2}\mathbf{y}, \frac{1}{2}\nu^{\star}}_t\Big)q(\d t) -\Phi^*_s,\,\d \Big(\frac{1}{2}\nu^{\star}_s\Big)\right\rangle. \nonumber
\end{align}
Hence, the fact that $\nu^{\star}\in \mathcal{C}$, and Lemma \ref{lemmafinito} yield
\begin{equation}
\label{integrabilityepsi-2}
-\infty < \E\left[\int_0^T\left\langle \int_s^Te^{(t-s)\mathcal{A}^*}\nabla \Pi\Big(t,Y^{\frac{1}{2}\mathbf{y}, \frac{1}{2}\nu^{\star}}_t\Big)q(\d t) - \Phi^*_s,\,\d \Big(\frac{1}{2}\nu_s^\star\Big)\right\rangle \right] < \infty.
\end{equation}

From \eqref{estimate-nec-1} we therefore obtain 
\begin{equation}\label{qqq2}
\begin{split}
&\E\left[\int_0^T\left\langle \int_s^Te^{(t-s)\mathcal{A}^*}\nabla \Pi(t,Y^\varepsilon_t)q(\d t) -\Phi^*_s,\, \d \nu^{\star}_s\right\rangle \right]\\
\geq & \, \E\left[\int_0^T\left\langle \int_s^Te^{(t-s)\mathcal{A}^*}\nabla \Pi(t,Y^\varepsilon_t)q(\d t) -\Phi^*_s,\, \d \nu_s\right\rangle \right].
\end{split}
\end{equation}
Now, on the one hand, Fatou's Lemma gives
\begin{equation}\label{fat}
\begin{split}
&\liminf_{\varepsilon\downarrow 0} \E\left[\int_0^T\left\langle \int_s^Te^{(t-s)\mathcal{A}^*}\nabla \Pi(t,Y^\varepsilon_t)q(\d t) -\Phi^*_s,\, \d \nu_s\right\rangle \right]\\
\geq & \, \E\left[\int_0^T\left\langle \int_s^Te^{(t-s)\mathcal{A}^*}\nabla \Pi(t,Y^\star_t)q(\d t) -\Phi^*_s,\, \d \nu_s\right\rangle \right].
\end{split}
\end{equation}
On the other hand, \eqref{integrabilityepsi-1} and \eqref{integrabilityepsi-2} allow to invoke the dominated convergence theorem when taking limits as $\varepsilon \downarrow 0$ and obtain
\begin{equation}\label{fat2}
\begin{split}
&\lim_{\varepsilon\downarrow 0}\E\left[\int_0^T\left\langle \int_s^Te^{(t-s)\mathcal{A}^*}\nabla \Pi(t,Y^\varepsilon_t)q(\d t) -\Phi^*_s,\, \d \nu^{\star}_s\right\rangle \right]\\= &\E\left[\int_0^T\left\langle \int_s^Te^{(t-s)\mathcal{A}^*}\nabla \Pi(t,Y^\star_t)q(\d t) -\Phi^*_s,\, \d \nu^{\star}_s\right\rangle \right]
\end{split}
\end{equation} 
Combining \eqref{qqq2} with \eqref{fat}-\eqref{fat2} provides 
\begin{eqnarray*}
0&\geq &\E\left[\int_0^T\left\langle \int_s^Te^{(t-s)\mathcal{A}^*}\nabla \Pi(t,Y^\star_t)q(\d t) -\Phi^*_s,\, \d \nu_s- \d \nu^{\star}_s\right\rangle \right]\\
&=& \E\left[\int_0^T\left\langle \E\left[\int_s^Te^{(t-s)\mathcal{A}^*}\nabla \Pi(t,Y^\star_t)q(\d t)\,\Big|\,\mathcal{F}_s\right] -\Phi^*_s,\, \d \nu_s- \d \nu^{\star}_s\right\rangle \right].
\end{eqnarray*}
The claim then follows recalling \eqref{Psistar} and by arbitrariness of $\nu \in \mathcal{C}$.
\vspace{0.25cm}

\emph{Step 2.} We now prove that the linear problem of the previous step has zero value; that is,
$$\sup_{\nu\in\mathcal{C}} \E\left[\int_0^T\langle \Psi^\star_t,\d \nu_t\rangle\right] =0,$$
for $\Psi^{\star}$ as in \eqref{Psistar}.

Clearly, by noticing that the admissible control $\nu\equiv0$ is a priori suboptimal, we have
$$\sup_{\nu\in\mathcal{C}} \E\left[\int_0^T\langle \Psi^\star_t,\d \nu_t\rangle\right] \geq 0.$$

To show the reverse inequality, we argue by contradiction, and we assume that there exists $t_o \in [0,T]$ such that $\esssup_{\Omega\times D} \,\Psi^{\star}_{t_o} >0$. Then, since $\Psi^\star$ is $\mathbb{F}$-adapted, there exist $\varepsilon>0$, $A\in \mathcal{M}$ with $\mu(A)>0$, and $E\in\mathcal{F}_{t_o}$ with $\P(E)>0$ such that 
$$\Psi^\star_{t_o}\geq\varepsilon \ \ \ \ \mbox{on} \  E \times A.$$
% with positive measure, and define
%$$\tau_{\varepsilon}(x):=\inf\big\{t\in[0,T]: \ \Psi^\star_t(x)\geq \varepsilon\big\}.$$
Consider the adapted, nondecreasing, nonnegative real-valued process
$$
\overline{\nu}_t:=\nu^{\star}_t + \mathds{1}_{E\times [t_o,T]\times A}.
$$
We clearly have that 
$$\E\left[\int_0^T\langle \Psi^\star_t,\d \overline{\nu}_t\rangle\right] = \E\left[\int_0^T\langle \Psi^\star_t,\d \nu^{\star}_t\rangle\right] +  \varepsilon \P(E)\mu(A) > \E\left[\int_0^T\langle \Psi^\star_t,\d \nu^{\star}_t\rangle\right],$$
thus contradicting that $\nu^{\star}$ is optimal for the linear problem. Hence, $\Psi^\star_t\leq 0$ for all $t\in [0,T]$, a.e.\ in $\Omega \times D$, and this gives that $\sup_{\nu\in\mathcal{C}} \E\left[\int_0^T\langle \Psi^\star_t,\d \nu_t\rangle\right]\leq0$.
\vspace{0.25cm}

\emph{Step 3.} The final claim follows by combinig Steps 1 and 2.
\end{proof}

\begin{remark}
\label{BVcase}
The proof of Theorem \ref{prop:FOC} hinges on the concavity of the running profit function with respect to the controlled state $Y^{\mathbf{y},\nu}$, and on the affine structure of the mapping $\nu \mapsto Y^{\mathbf{y},\nu}$. It is then reasonable to expect that one might derive necessary and sufficient first-order conditions for optimality also when $Y^{\mathbf{y},\nu}$ evolves as in \eqref{eq:abstract}, but $t \mapsto \nu_t$ is a process with paths of (locally) bounded variation. In such a case, our approach still applies by identifying each admissible control $\nu$ with a (random) signed countably additive vector measure $\nu: \mathcal{B}([0,T]) \to X$ of finite variation.
\end{remark}

%%%%%%%%%%%%%%%%%%%%%%%%%%%%%%%%%%%%%%%%%%

\section{The case in which $\mathcal{A}$ generates a group}
\label{sec:group}

In this section we will consider the case when $\mathcal{A}$ generates a $C_0$-\emph{group} of  {positivity-preserving  operators}. In this case, since for any given $t\geq0$ we can define the inverse $e^{-t\A}$, the controlled dynamics \eqref{mild} takes the separable form 
\begin{equation}
\label{mild2}
Y_t^{\mathbf{y},\nu}=e^{t\A}\Big[\mathbf{y}+ \widehat{\nu}_t\Big] = e^{t\A} \widehat{Y}_t^{\mathbf{y},\widehat{\nu}}, \qquad Y^{\mathbf{y},\nu}_{0^-}=\mathbf{y},
\end{equation}
where, for any $\nu \in \mathcal{S}$, we have set 
\begin{equation}
\label{nuhat}
\widehat{\nu}_t:=\int_0^t e^{-s\A}\, \d \nu_s, \qquad \widehat{\nu}_{0^-}=0,
\end{equation}
and 
\begin{equation}
\label{Yhat}
\widehat{Y}_t^{\mathbf{y},\widehat{\nu}}:= \mathbf{y}+ \widehat{\nu}_t, \quad \widehat{Y}_{0^-}^{\mathbf{y},\widehat{\nu}}=\mathbf{y}.
\end{equation}
Notice that \eqref{mild2} is formally equivalent to the expression of the controlled dynamics that one would have in a one-dimensional setting where the process $Y$ is affected linearly by a monotone control and depreciates over time at a constant rate (see, e.g., p.\ 770 in \cite{BankRiedel1} or eq.\ (2.3) in \cite{CFS}).

Letting 
$$\widehat{\mathcal{C}}:=\Big\{\widehat{\nu}\in \mathcal{S}:\, \E\left[\int_0^T \langle e^{t\mathcal{A}^*}\Phi^*_t,\,\d \widehat{\nu}_t\rangle \right] < \infty\Big\},$$
we notice that the mapping $\mathcal{C}\to\widehat{\mathcal{C}}, \ \nu\mapsto \widehat{\nu}$, is one-to-one and onto. In particular, for any $\widehat{\nu} \in \widehat{\mathcal{C}}$ one has that $\nu_t:=\int_0^t e^{s\A}\, \d \widehat{\nu}_s \in \mathcal{C}$.
As a consequence, for any $\mathbf{y} \in K_+$, problem \eqref{OCproblem} reads
\begin{equation}
\label{OCP-group}
v(\mathbf{y}) = \sup_{\widehat{\nu} \in \widehat{\mathcal{C}}}\E\left[\int_0^T \Pi\left(t, e^{t\mathcal{A}} \widehat{Y}_t^{\mathbf{y},\widehat{\nu}} \right)q(\d t) -  \int_0^T \langle e^{t\mathcal{A}^*}\Phi^*_t,\,\d \widehat{\nu}_t\rangle \right].
\end{equation}

Theorem \ref{prop:FOC} can then be reformulated as follows.

\begin{corollary}
\label{cor:FOC-group} 
A control $\widehat{\nu}^{\star}\in\widehat{\mathcal{C}}$ is optimal for problem \eqref{OCP-group} {if and only if} the following First-Order Conditions (FOCs) hold true:
\begin{enumerate}[(i)]
\item For every $\widehat{\nu}\in \widehat{\mathcal{C}}$
$$\E\left[\int_0^T\Big\langle e^{s\mathcal{A}^*}\E\left[\int_s^Te^{(t-s)\A^*} \nabla \Pi\left(t,e^{t\A}\widehat{Y}_t^{\mathbf{y},\widehat{\nu}^{\star}}\right)q(\d t)\,\Big|\,  \mathcal{F}_s\right]-e^{s\mathcal{A}^*}\Phi^*_s,\d \widehat{\nu}_s\Big\rangle  \right]\leq 0;$$
\item the following equality holds:
 $$\E\left[\int_0^T\Big\langle e^{s\mathcal{A}^*}\E\left[\int_s^T e^{(t-s)\A^*} \nabla \Pi\left(t,e^{t\A}\widehat{Y}_t^{\mathbf{y},\widehat{\nu}^{\star}}\right)q(\d t)\,\Big|\,  \mathcal{F}_s\right]-e^{s\mathcal{A}^*}\Phi^*_s,\d \widehat{\nu}^{\star}_s\Big\rangle  \right] = 0.$$
\end{enumerate}
\end{corollary}

The previous discussion (cf.\ \eqref{mild2}, \eqref{nuhat}, \eqref{Yhat}) and Corollary \ref{cor:FOC-group} immediately yield the following.
\begin{proposition}
\label{lemm:optimal}
Suppose that $\widehat{\nu}^{\star}\in\widehat{\mathcal{C}}$ is an optimal control for problem \eqref{OCP-group}. Then, $\nu^{\star}_t:=\int_0^t e^{s\A}\, \d \widehat{\nu}^{\star}_s \in \mathcal{C}$ is an optimal control for problem \eqref{OCproblem} and $\big(e^{t\A}\widehat{Y}^{\mathbf{y},\widehat{\nu}^{\star}}_t\big)_{t \geq 0}$ is its associated optimally controlled state process.
\end{proposition}

\begin{remark}
\label{rem:existence}
We have obtained necessary and sufficient conditions for optimality for the {\rm{concave}} problem of maximization of an expected net profit functional. Through the same arguments employed above, a similar characterization of the optimal control can be obtained  for the {\rm{convex}} problem of minimization of a total expected cost functional of the form
$$\E\left[\int_0^T C\left(t, Y_t^{\mathbf{y},\nu}\right) q(\d t) +  \int_0^T \langle \Phi^*_t,\,\d \nu_t\rangle \right].$$
Here, $C:\Omega \times [0,T] \times K_+ \to \mathbb{R}_+$, $(\omega,t,\mathbf{k}) \mapsto C(\omega,t,\mathbf{k})$ is convex with respect to $\mathbf{k}$, and satisfies suitable additional technical requirements.

Within such a setting, suppose that $X=L^2(D)$ and identify, through the usual Riesz representation, $X^*=X$. Assume that $\mathcal{A}$ generates a $C_0$-group of {positivity-preserving operators} and that $\langle e^{t\A} \mathbf{k}, e^{t\A} \mathbf{k} \rangle \geq m_t \langle \mathbf{k}, \mathbf{k} \rangle$ for a suitable strictly positive function $m$ and for all $\mathbf{k} \in K_+$ {(the latter condition is verified, e.g., with $m\equiv 1$ if $\A$ is a skew-adjoint operator)}. Then, it is possible to show that, taking $C(\omega,t,\mathbf{k}) := \langle \mathbf{k} - Z_t(\omega), \mathbf{k} - Z_t(\omega) \rangle$, for a suitably integrable $X$-valued $\mathbb{F}$-adapted stochastic process $Z$, there exists an optimal control for the cost minimization problem. This is due to the fact that the previous specifications of the problem's data allow to prove that any minimizing sequence is uniformly bounded in $L^2(\Omega \times [0,T]; L^2(X))$, so that standard arguments may be used to show the existence of an optimizer.

A similar strategy seems not to be feasible if one aims at proving existence of an optimizer in the current studied case of the maximization of a net profit functional, where typically the running profit function grows at most linearly. However, in the next section, under suitable requirements on the problem's data and the standing assumption of this section that $\mathcal{A}$ generates a $C_0$-group of {positivity-preserving operators}, we provide the explicit expression of an optimal control.
\end{remark}

%%%%%%%%%%%%%%%%%%%%%%%%%%%%%%%%%%%%%%%%%

\subsection{Explicit solution in a separable setting}
\label{sec:explicit}

We now provide an explicit solution to problem \eqref{OCproblem} in a specific separable context. The following study is motivated by a problem of irreversible investment as outlined in the introduction.

Consider a globally operating company that can invest irreversibly into capacity of local sub-companies at locations $x\in D$. We assume that $D$ is equipped with a suitable finite measure $\mu$. At each location, the same product is produced and sold at the global market for a stochastic, time-varying price. If the company operates at decreasing returns to scale, total  profit at time $t$ when capacity at location $x$ is $Y_t(x)$ can be written in the form
$$ (z^*_t)^\alpha  \int_D \left(Y_t(x)\right)^{1-\alpha} \mu(\d x), \qquad \alpha \in (0,1),$$ 
where the stochastic process $z^*$ is derived from the (global, stochastic) output price and wages. We also assume that the cost of investment into capacity does not depend on the specific location $x$; think, again, of a globally traded input like labor, technology etc.\ that has a globally uniform price $\varphi^*_t$. The operator $\A$ describes the impact of a firm on its neighbors; these could be spillover effects of investments, demographic changes, labor mobility etc.
 
This irreversible investment problem falls into the following class of problems.
\begin{assumption}
\label{ass:group}
\begin{itemize}\hspace{10cm}
\item[(i)] $\mu(D)<\infty$, $T=+\infty$, and $q(\d t)= e^{-rt}dt$, for some $r>0$. 
\item[(ii)] The unitary vector $\mathbf{1}$ is an eigenvector of $\mathcal{A}$ and $\mathcal{A}^*$ with associated eigenvalues $\lambda_0 \in \mathbb{R}$ and $\lambda_0^* \in \mathbb{R}$, respectively.
\item[(iii)] $r>\lambda_0^* \vee 0$.
\item[(iv)] $\Pi(\omega,t,\mathbf{k})= (1-\alpha)^{-1}(z_t^*)^{\alpha}(\omega)\langle \mathbf{1},\mathbf{k}^{1-\alpha}(\cdot)\rangle$, $(\omega,t,\mathbf{k})\in \Omega \times \mathbb{R}_+ \times K_+$, for some $\alpha \in (0,1)$ and for an $\mathbb{F}$-progressively measurable nonnegative process $(z_t^*)_{t\geq0}$.
\item[(v)] $\Phi^*_t(\omega)= e^{-rt}\varphi^*_t(\omega)\mathbf{1}$, for all $(\omega,t)\in \Omega \times \mathbb{R}_+$, for an $\mathbb{F}$-progressively measurable, nonnegative, c\`adl\`ag  process $(\varphi^*_t)_{t\geq0}$ such that $(e^{-(r-\lambda_0^*)t} \varphi^*_{t})_{t\geq0}$ is of Class (D), lower-semicontinuous in expectation, and $\limsup_{t \uparrow \infty}e^{-(r-\lambda_0^*)t} \varphi^*_{t}=0$. 
\end{itemize}
\end{assumption}

Notice that although the process $\Phi^*$ is space-homogeneous, the problem is still space-inhomogeneous, since the initial distribution of the production capacity $\mathbf{y}$ does not need to be uniform. Under Assumption \ref{ass:group}, it holds
\begin{equation}
\label{example-1}
e^{t\A}\mathbf{1} = e^{\lambda_0 t}\mathbf{1} \quad \text{and} \quad e^{t\A^*} \mathbf{1}=e^{\lambda_0^* t}\mathbf{1}, \quad t \geq0,
\end{equation}
and we have
\begin{equation}
\label{example-gradient}
\nabla \Pi(\omega,t,\mathbf{k}) = (z^*_t(\omega))^{\alpha}\mathbf{k}(\cdot)^{-\alpha}, \quad (\omega,t,\mathbf{k})\in \Omega \times \mathbb{R}_+ \times K_+.
\end{equation}

\begin{remark}
\label{rem:operators}
Operators $\mathcal{A}$ and spaces $(D,\mu)$ satisfying Assumptions \ref{ass:A}, \ref{ass:group}-(i), and \ref{ass:group}-(ii) are, for instance:
\begin{itemize}
\item[(a)]
the space $D=S^1$ with the Hausdorff measure, where $S^1$ is the unit circle in $\R^2$, and the spatial-derivative operator $\mathcal{A}:=\frac{\d}{\d x}$ {with domain $W^{1,2}(S^1)$} on the space $X=L^2(S^1)$. In this case $\lambda_0=\lambda_0^* =0$; 
 \item[(b)]
any finite measure space $D$ and the integral operator
$$(\mathcal{A}f)(x):=\int_{D} a(x,y)f(y)\mu(\d y), \ \ \ \ f\in X:=L^2(D),$$ for a kernel $a \in L^2(D\times D)$ {with $a \geq 0$}, such that 
$$
\int_D a(x,y)\mu(\d y)=c_1 \ \ \ \forall x\in D; \ \ \ \  \int_D a(x,y)\mu(\d x)=c_2 \ \ \ \forall y\in D.
$$
In this case $\lambda_0=c_1$ and $\lambda_0^* =c_2$. Note that $\mathcal{A} \in \mathcal{L}^+(L^2(D))$,  thus 
$$ e^{t\mathcal{A}}f = \sum_{n=0}^{\infty} \frac{t}{n!}\mathcal{A}^n f.$$
Hence,  $(e^{t\A})_{t\geq 0}\subseteq \mathcal{L}^+(L^2(D))$.
\end{itemize}
\end{remark}

Before moving on with our analysis we need the following result. Its proof follows from a suitable application of the Bank-El Karoui's Representation Theorem (cf.\ Theorem 3 in \cite{BankElKaroui}).
\begin{lemma}
\label{lem:existenceBEK}
There exists a unique (up to indistinguishability) strictly positive optional solution $\ell=(\ell_t)_{t\geq 0}$ to 
\begin{equation}
\label{backwardeq-1dim}
\E\bigg[\int_{\tau}^{\infty} e^{-(r-\lambda_0^*)t} (z_t^*)^{\alpha}\,\Big(e^{\lambda_0 t}\sup_{\tau \leq u \leq t}e^{-\lambda_0 u}\ell_u\Big)^{-\alpha} \d t \bigg| \mathcal{F}_{\tau} \bigg] = e^{-(r-\lambda_0^*)\tau}\varphi^*_{\tau},
\end{equation}
for any $\mathbb{F}$-stopping time $\tau$.\footnote{We adopt the convention $e^{-(r-\lambda_0^*)\tau}\varphi^*_{\tau}:=\limsup_{t \uparrow \infty}e^{-(r-\lambda_0^*)t}\varphi^*_{t}$ on the event $\{\tau=+\infty\}$.}

Moreover, the process $\ell$ has upper right-continuous paths, and it is such that 
\begin{equation}
\label{eq:integrabilityBEK}
e^{-(r-\lambda_0^*)t} (z^*_t)^{\alpha} \Big(e^{\lambda_0 t}\sup_{s \leq u \leq t}e^{-\lambda_0 u}\ell_u \Big)^{-\alpha} \in L^1(\P \otimes \d t),
\end{equation}
for any $s \geq 0$.
\end{lemma}
\begin{proof}
Apply the Bank-El Karoui's Representation Theorem (cf.\ \cite{BankElKaroui}, Theorem 3) to (according to the notation of that paper)
\begin{equation}
\label{identification}
X_t(\omega):= e^{-(r-\lambda_0^*)t} \varphi^*_{t},
\end{equation}
and
\begin{equation}
\label{identificationf}
f(\omega,t,\ell):=
\left\{
\begin{array}{ll}
e^{-(r-\lambda_0^*)t} (z^*_t(\omega))^{\alpha} \left(\frac{e^{\lambda_o t}}{-\ell}\right)^{-\alpha},\,\,\,\,\,\mbox{for}\,\,\ell<0,\\ \\
- e^{-(r-\lambda_0^*)t} \ell \,,\,\,\,\,\,\,\,\,\,\,\,\,\,\,\,\,\,\,\,\,\,\,\,\,\,\,\,\,\,\,\,\,\,\,\,\,\,\,\,\,\,\,\,\,\,\mbox{for}\,\,\ell\geq 0.
\end{array}
\right.
\end{equation}
Indeed, defining
\begin{equation}
\label{definizionexil}
\Xi^{\ell}_t:= \essinf_{\tau \geq 0}\ \E\bigg[\int_t^{\tau} f(s,\ell) \d s + X_{\tau}\,\Big|\,\mathcal{F}_t\bigg],\qquad \ell \in \mathbb{R},\,\,\,t \geq 0,
\end{equation}
the optional process (cf.\ \cite{BankElKaroui}, eq.\ (23) and Lemma 4.13)
\begin{equation}
\label{defxi1}
\xi_t := \sup\big\{\ell \in \mathbb{R}:\,\,\Xi^{\ell}_t = X_t\big\},\qquad t\geq 0,
\end{equation}
solves the representation problem
\begin{equation}
\label{representationproblem0}
\E\bigg[\,\int_{\tau}^{T} f(s,\sup_{\tau \leq u \leq s} \xi_u)\,\d s\,\Big|\,\mathcal{F}_{\tau}\,\bigg] = X_{\tau},
\end{equation}
for any $\mathbb{F}$-stopping time $\tau$.

If now $\xi$ has upper right-continuous paths and it is strictly negative, then the strictly positive, upper right-continuous process $\ell_t = - \frac{e^{-\lambda_o t}}{\xi_t}$ solves
\begin{eqnarray*}
\label{representationproblem2}
e^{-(r-\lambda_0^*)\tau}\varphi^*_{\tau} &\hspace{-0.25cm} = \hspace{-0.25cm} &\E\bigg[\,\int_{\tau}^{\infty} e^{-(r-\lambda_0^*)t} (z^*_t)^{\alpha}\left(\frac{e^{\lambda_o t}}{-\sup_{\tau \leq u \leq t}( - \frac{e^{\lambda_o u}}{\ell_{u}})}\right)^{-\alpha}\,\d t\,\Big|\,\mathcal{F}_{\tau}\,\bigg]\nonumber \\
&\hspace{-0.25cm} = \hspace{-0.25cm}&\E\bigg[\,\int_{\tau}^{\infty} e^{-(r-\lambda_0^*)t} (z^*_t)^{\alpha}\big(e^{\lambda_o t} \sup_{\tau \leq u \leq t} e^{-\lambda_o u}\ell_{u}\big)^{-\alpha}\,\d t\,\Big|\,\mathcal{F}_{\tau}\,\bigg], \nonumber
\end{eqnarray*}
for any $\mathbb{F}$-stopping time $\tau$; i.e.\ $\ell$ solves \eqref{backwardeq-1dim}, thanks to \eqref{identificationf} and \eqref{representationproblem0}. Moreover, $\xi$ (and hence $\ell$) is unique up to optional sections by \cite{BankElKaroui}, Theorem $1$, as it is optional and upper right-continuous. Therefore it is unique up to indistinguishability by Meyer's optional section theorem (see, e.g., \cite{DM}, Theorem IV.86).

To complete the proof, we must show that $\xi$ is indeed upper right-continuous and strictly negative. This can be done by following the arguments employed at the end of the proof of Proposition 3.4 of \cite{Ferrari15}.
\end{proof}

Notice that equation \eqref{backwardeq-1dim} might be explicitly solved when the processes $z^*$ and $\varphi^*$ are specified. If, for example, $z^*$ and $\varphi^*$ are exponential L\'evy processes, then so it is the ratio $\eta^*:= (z^*)^{\alpha}/\varphi^*$. In this case, \eqref{backwardeq-1dim} is equivalent to the backward equation solved by Riedel and Su in Proposition 7 of \cite{RiedelSu} once we set, in their notation, $\eta^*:=X$ and $\lambda_0^*:=-\delta$. Within this specification, $\ell_{t}=\kappa \eta^*_{t}$, for a suitable constant $\kappa>0$ that can be explicitly determined with the help of the Wiener-Hopf factorization. We also refer to \cite{FerrariSalminen} for related results in a L\'evy setting, and to \cite{Ferrari15} for explicit solutions when the underlying randomness is driven by a one-dimensional regular diffusion. 

We now provide a possible sufficient condition on the processes $z^*$ and $\varphi^*$ ensuring that the performance criterion $\mathcal{J}(\mathbf{y}, \nu)<\infty$ is finite for any admissible control $\nu$ (cf.\ Assumption \ref{ass:U}-(ii)).
\begin{lemma}
\label{lem:finitefunctional}
Suppose that
\begin{equation}
\label{eq:integr}
\E\bigg[\int_0^{\infty} e^{-(r \wedge (r-\lambda_0^*))t} (z^*_t)^{\alpha} dt \bigg] < \infty,
\end{equation}
and that there exists $m>0$ such that
\begin{equation}
\label{cond:finiteJ}
\E\bigg[\int_s^{\infty} e^{-(r-\lambda_0^*)t} (z^*_t)^{\alpha} dt \,\Big|\, \mathcal{F}_s\bigg] \leq m e^{-(r-\lambda_0^*)s}\varphi^*_s,
\end{equation}
for all $s\geq0$.
Then, there exists $C>0$, independent of $\mathbf{y}$, such that $v(\mathbf{y}) \leq C\big(1 + \langle \mathbf{1}, \mathbf{y} \rangle\big)$. 
\end{lemma}
\begin{proof}
Recall that any $\nu \in \mathcal{C}$ defines $\widehat{\nu}\in \widehat{\mathcal{C}}$ through \eqref{nuhat}. Since $\alpha\in (0,1)$, for any $\varepsilon >0$ there exists $\kappa_{\varepsilon}>0$ such that, for all $\nu \in \mathcal{C}$ and $t\geq0$, one has 
\begin{eqnarray*}
& \displaystyle \big \langle \mathbf{1}, \big(e^{t\A}(\mathbf{y} + \widehat{\nu}_t)\big)^{1-\alpha} \big \rangle \leq \kappa_{\varepsilon} \langle \mathbf{1},\mathbf{1} \rangle + \varepsilon \big \langle \mathbf{1}, e^{t\A}\big(\mathbf{y} + \widehat{\nu}_t\big)\big \rangle = \kappa_{\varepsilon} \mu(D) + \varepsilon e^{\lambda_0^* t} \langle \mathbf{1} , \mathbf{y} \rangle + \varepsilon e^{\lambda_0^* t} \langle \mathbf{1}, \widehat{\nu}_t \rangle.
\end{eqnarray*}
Then, by the latter estimate, we have
\begin{eqnarray}
\label{eq:finiteJ}
&& \hspace{0.5cm} \mathcal{J}(\mathbf{y}, \nu) = \E\bigg[\int_0^{\infty} e^{-rt} (1-\alpha)^{-1}(z^*_t)^{\alpha} \big \langle \mathbf{1}, \big(e^{t\A}(\mathbf{y} + \widehat{\nu}_t)\big)^{1-\alpha} \big \rangle \d t- \int_0^{\infty} \langle e^{t\A^*}\Phi^*_t, \d \widehat{\nu}_t \rangle\bigg]  \\
&& \leq \kappa_{\varepsilon}C_1 + \varepsilon C_2(\mathbf{y}) + \varepsilon \E\bigg[\int_0^{\infty} e^{-(r-\lambda_0^*)t} (z^*_t)^{\alpha} \Big \langle \mathbf{1}, \int_0^t \d \widehat{\nu}_s \Big \rangle \d t\bigg] - \E\bigg[\int_0^{\infty} e^{-(r-\lambda_0^*)s} \varphi^*_s \langle \mathbf{1}, \d \widehat{\nu}_s \rangle\bigg],\nonumber
\end{eqnarray}
where we have set
$$C_1:= \mu(D) \E\bigg[\int_0^{\infty} e^{-rt} (z^*_t)^{\alpha} \d t\bigg] \quad \text{and} \quad C_2(\mathbf{y}):= \langle \mathbf{1}, \mathbf{y} \rangle \E\bigg[\int_0^{\infty} e^{-(r-\lambda_0^*)t} (z^*_t)^{\alpha} \d t\bigg].$$
Notice that $C_1$ and $C_2(\mathbf{y})$ are finite due \eqref{eq:integr}.
By employing now Lemma \ref{lemma:fubini} and \eqref{DM-opt} in the first expectation in the last line of \eqref{eq:finiteJ} we then obtain for any $\varepsilon>0$
\begin{eqnarray*}
\label{eq:finiteJ-2}
&& \mathcal{J}(\mathbf{y}, \nu) \leq \kappa_{\varepsilon}C_1 + \varepsilon C_2(\mathbf{y}) + \E\bigg[\int_0^{\infty} \bigg( \varepsilon \E\bigg[\int_s^{\infty} e^{-(r-\lambda_0^*)t} (z^*_t)^{\alpha} \d t \Big| \mathcal{F}_s \bigg] - e^{-(r-\lambda_0^*)s} \varphi^*_s\bigg) \langle \mathbf{1}, \d \widehat{\nu}_s \rangle\bigg] \nonumber \\
&& \leq \kappa_{\varepsilon}C_1 + \varepsilon C_2(\mathbf{y}) + \E\bigg[\int_0^{\infty} e^{-(r-\lambda_0^*)s} \varphi^*_s \big( \varepsilon m - 1 \big) \langle \mathbf{1}, \d \widehat{\nu}_s \rangle\bigg], 
\end{eqnarray*}
where \eqref{cond:finiteJ} has been used in the last step. The thesis finally follows by taking $\varepsilon < \frac{1}{m}$, and then the supremum over $\nu\in\mathcal{C}$.
\end{proof}

\begin{remark}
Notice that \eqref{cond:finiteJ} is satisfied, e.g., if $z^*$ and $\varphi^*$ are exponential L\'evy processes,  $\E\big[\int_0^{\infty}e^{-(r-\lambda_0^*)t} (z^*_t)^{\alpha} \d t \big] < \infty$, and there exists $m>0$ such that 
$$ (z^*_s)^{\alpha}\cdot  \E\left[\int_0^{\infty}e^{-(r-\lambda_0^*)t} (z^*_t)^{\alpha} \d t \right]  \leq m \varphi^*_s, \quad \forall s \geq 0.$$
\end{remark}

From now on we assume that the (sufficient) conditions \eqref{eq:integr} and \eqref{cond:finiteJ} hold. Then, thanks to Lemma \ref{lem:finitefunctional}, we fulfill Assumption \ref{ass:U} and our first-order conditions approach can be applied, yielding the following result. 

\begin{proposition}
\label{example-verification}
Let $\ell_0$ be the initial value of the process $(\ell_t)_{t\in[0,T]}$ of Lemma \ref{lem:existenceBEK}, suppose $\mathbf{y}\leq \ell_0 \mathbf{1}$ and consider the nondecreasing $\mathbb{F}$-adapted, $K_+$-valued, right-continuous process
\begin{equation}
\label{example-OC-group}
\widehat{\nu}^{\star}_t:=\mathbf{1}\sup_{0 \leq u \leq t}e^{-\lambda_0 u} \ell_u - \mathbf{y}, \qquad \widehat{\nu}^{\star}_{0^-}=0.
\end{equation}
Then, $\widehat{\nu}^{\star}$ is optimal for problem \eqref{OCP-group} if $\displaystyle \E\Big[\int_0^{\infty} \langle e^{t\A^*}\Phi^*_t, \d \widehat{\nu}^{\star}_t \rangle \Big] < \infty$.
\end{proposition}

\begin{proof}
First of all, notice that from \eqref{mild2} and \eqref{example-OC-group} we can write for any $t \geq 0$ 
\begin{equation}
\label{prod-cap-homo}
\widehat{Y}_t^{\mathbf{y},\widehat{\nu}^{\star}} = \mathbf{1} \sup_{0 \leq u \leq t} e^{-\lambda_0 u} \ell_u =:\widehat{Y}_t^{\star}.
\end{equation}

To prove the optimality of \eqref{example-OC-group} for problem \eqref{OCP-group} it suffices to verify that such an admissible control verifies the first-order conditions for optimality of Corollary \ref{cor:FOC-group}. 

By monotonicity of $\mathbf{k} \mapsto \nabla \Pi(t,\mathbf{k})$, we then have from \eqref{example-1}, \eqref{example-gradient}, and Lemma \ref{lem:existenceBEK}
\begin{align*}
& \E\left[\int_0^\infty\Big \langle e^{s\A^*} \E\left[\int_s^\infty e^{(t-s)\A^*} \nabla \Pi\left(t, e^{t\A}\widehat{Y}_t^{\star}\right) q(\mathrm{d}t)\,\Big|\,  \mathcal{F}_s\right]- e^{s\A^*}\Phi^*_s, \d \widehat{\nu}_s\Big\rangle  \right] \nonumber \\
& {=}  \E\left[\int_0^\infty  e^{-(r-\lambda_0^*)s} \Big \langle  \mathbf{1} \Big(\E\left[\int_s^\infty e^{-(r-\lambda_0^*) (t-s)} (z^*_t)^{\alpha} \Big(e^{\lambda_0 t}\sup_{0 \leq u \leq t} e^{-\lambda _0 u}\ell_u\Big)^{-\alpha}\mathrm{d}t\,\Big|\,  \mathcal{F}_s\right]- \varphi^*_s\Big), \d \widehat{\nu}_s \Big\rangle  \right] \\
& \leq  \E\left[\int_0^\infty  e^{-(r-\lambda_0^*)s} \Big \langle  \mathbf{1} \Big(\E\left[\int_s^\infty e^{-(r-\lambda_0^*) (t-s)} (z^*_t)^{\alpha} \Big(e^{\lambda_0 t}\sup_{s \leq u \leq t} e^{-\lambda _0 u}\ell_u\Big)^{-\alpha}\mathrm{d}t\,\Big|\,  \mathcal{F}_s\right]- \varphi^*_s\Big), \d \widehat{\nu}_s \Big\rangle  \right] =0. \nonumber
\end{align*}
Hence, the inequality in claim (i) of Corollary \ref{cor:FOC-group} is satisfied by $\widehat{\nu}^{\star}$.

In order to prove that the equality in claim (ii) of Corollary \ref{cor:FOC-group} holds, notice that $\widehat{\nu}^{\star}_0(x)=\mathbf{1}(x)\ell_0 - \mathbf{y}(x)$ and $\widehat{Y}_0^{\star}=\mathbf{1}(x)\ell_0$, $x \in D$. Moreover, we have that the times of increase of $\widehat{\nu}^\star_{\cdot}(x)$ on $(0,\infty)$ (i.e.\ any strictly positive time in the support of the measure induced on $\mathbb{R}_+$ by the nondecreasing process $t \mapsto \widehat{\nu}^{\star}_t(\omega,x)$, $(\omega,x)\in \Omega \times D$) are independent of $x \in D$ since, by \eqref{example-OC-group}, they coincide with the time of increase of the nondecreasing process $\zeta^{\star}_{\cdot}:=\sup_{0 \leq u \leq \cdot} e^{-\lambda_0 u} \ell_u$ which is independent of $x$. 
At any of such times $s>0$ we have
\begin{equation}
\label{eq:binding}
\big[e^{t\A}\widehat{Y}_t^{\star}\big](x) = \left[e^{t\A}\mathbf{1} \sup_{0 \leq u \leq t} e^{-\lambda_0 u} \ell_u\right](x) = \left[e^{\lambda_0 t}\sup_{s \leq u \leq t} e^{- \lambda_0 u} \ell_u\right], \quad \forall x \in D.
\end{equation}

Therefore, thanks to the previous considerations, we have $\d\widehat{\nu}^{\star}_t(x) = \mathbf{1}(x) \d\zeta^{\star}_t$ for all $x \in D$ and $t > 0$, and, together with Lemma \ref{lem:existenceBEK}, this allows to write
\begin{align}
\label{example-verifying2}
& \E\left[\int_0^\infty\Big \langle e^{s\A^*} \E\left[\int_s^\infty e^{(t-s)\A^*} \nabla \Pi\left(t,Y_t^{\star}\right) q(\mathrm{d}t)\,\Big|\,  \mathcal{F}_s\right]- e^{s\A^*}\Phi^*_s, \d \widehat{\nu}^{\star}_s\Big\rangle  \right] \nonumber \\
& = \E\left[\int_0^\infty\Big \langle \E\left[\int_0^\infty e^{t\A^*} \nabla \Pi\left(t,Y_t^{\star}\right) q(\mathrm{d}t)\right]- \Phi^*_0, \d \widehat{\nu}^{\star}_0\Big\rangle  \right] \nonumber \\
& + \E\left[\int_0^\infty\Big \langle e^{s\A^*} \E\left[\int_s^\infty e^{(t-s)\A^*} \nabla \Pi\left(t,Y_t^{\star}\right) q(\mathrm{d}t)\,\Big|\,  \mathcal{F}_s\right]- e^{s\A^*}\Phi^*_s, \d \widehat{\nu}^{\star}_s\Big\rangle  \right] \nonumber \\
& = \E\left[ \int_{D} \Big(\E\left[\int_0^\infty e^{-(r-\lambda_0^*) t} (z^*_t)^{\alpha} \left(e^{\lambda_0 t}\sup_{0 \leq u \leq t} e^{-\lambda _0 u}\ell_u\right)^{-\alpha}\mathrm{d}t\right]- \varphi^*_0\Big) \big(\ell_0 - \mathbf{y}(x)\big) \mu(dx) \right] \nonumber \\
& + \mu(D) \E\left[\int_{0^+}^\infty e^{-(r-\lambda_0^*)s} \left(\E\left[\int_s^\infty e^{-(r-\lambda_0^*) (t-s)} (z^*_t)^{\alpha} \Big(e^{\lambda_0 t}\sup_{s \leq u \leq t} e^{-\lambda _0 u}\ell_u\Big)^{-\alpha}\mathrm{d}t\,\Big|\,  \mathcal{F}_s\right]- \varphi^*_s\right) \d \zeta^{\star}_s \right] =0, \nonumber 
\end{align}
thus completing the proof of the optimality of $\widehat{\nu}^{\star}$.
\end{proof}

\begin{remark}
\label{rem:jumps-integr}
\begin{itemize}
\item[(a)] Notice that one can identify two different kinds of jumps in the optimal control $\widehat{\nu}^{\star}$. At initial time, a lump sum investment of size $\mathbf{1}(x)\ell_0 - \mathbf{y}(x)$ at position $x \in D$ allows to instantaneously move to the initial uniform desired level $\ell_0$, and thus resolve an initial situation of under-production. Notice that at each position $x\in D$ a different size of investment should be optimally made. Subsequent jumps of $\widehat{\nu}^{\star}$ are instead due only to the possible jumps of the stochastic time-dependent target $\ell$. Those jumps are typically related to those of the processes $z^*$ and $\varphi^*$, and therefore we can think of them as lump sum adjustments in the production capacity due to shocks in the market, e.g.\ shocks in market's demand of the produced good or in the marginal price of investment. Those jumps do not affect the space distribution of the production capacity, which in fact remains uniform.     

\item[(b)] The integrability condition on $\widehat{\nu}^{\star}$ required in Proposition \ref{example-verification} has to be verified on a case by case basis when explicit solutions to \eqref{backwardeq-1dim} are available (see, for example, Theorem 7.2 in \cite{RiedelSu} or Section 4 of \cite{FerrariSalminen} for L\'evy settings). Generally speaking, if one picks $r$ sufficiently large then it can be shown that 
\begin{align*}
& \E\bigg[\int_0^{\infty} \langle e^{t\A^*}\Phi^*_t, \d \widehat{\nu}^{\star}_t \rangle\bigg] = \E\bigg[\int_0^{\infty} e^{-(r-\lambda_0^*) t} \varphi^*_t \langle \mathbf{1} , \d \widehat{\nu}^{\star}_t \rangle\bigg]  \nonumber \\
& = \int_D \big(\mathbf{1}(x) \ell_0 - \mathbf{y}(x)\big) \mu(\d x) + \mu(D) \E\bigg[\int_{0^+}^{\infty} e^{-(r-\lambda_0^*) t} \varphi^*_t \d \big(\sup_{0\leq u \leq t} e^{-\lambda_0 u} \ell_u \big) \bigg] < \infty. \nonumber
\end{align*}
\end{itemize}
\end{remark}

Propositions \ref{example-verification} and \ref{lemm:optimal} then yield the next result.

\begin{corollary}
\label{cor:optimalexample}
Let $\mathbf{y} \leq \mathbf{1}\ell_0$ and lett $\widehat{\nu}^{\star}$ of \eqref{example-OC-group} ibe such that $\E[\int_0^{\infty} \langle e^{t\A^*}\Phi^*_t, \d \widehat{\nu}^{\star}_t \rangle] < \infty$. Then, 
$$\nu^{\star}_{t} := \int_0^{t} e^{s\A} \d \widehat{\nu}^{\star}_s, \quad \nu^{\star}_{0^-}=0,$$
is optimal for problem \eqref{OCproblem} and
$$Y_t^{\star}:= \mathbf{1} e^{\lambda_0 t} \sup_{0 \leq u \leq t} e^{-\lambda_0 u} \ell_u, \quad Y^{\star}_{0^-}=\mathbf{y}$$ 
is the optimally controlled production capacity.
\end{corollary}

%%%%%%%%%%%%%%%%%%%%%%%%%%%%%%%%%%%%%%%%%%%%%%%%%%%%%%%%%%%%%%%%%%%%%

\section{Application to PDE models}
\label{Sec:PDE}

In this section we consider PDE frameworks in which the requirements in Assumption \ref{ass:A} on the abstract operator $\A$ are fulfilled. In particular, we discuss the case in which $\mathcal{A}$ is an elliptic self-adjoint operator, and we illustrate two possible cases having potential applications.

\subsection{Dirichlet boundary conditions in the $n$-dimensional space}
\label{sec:Dirichlet}

Let $D\subseteq \R^n$ be an open domain. Consider $a_{ij}: D \to \mathbb{R}$, $1 \leq i,j\leq d$, $f:D \to \mathbb{R}$ bounded and Borel-measurable such that for some $\lambda>0$
\begin{equation}
\label{eq:unellip}
\sum_{i,j=1}^n a_{ij}(x) \xi_i\xi_j \geq \lambda |\xi|^2, \quad \forall x \in D,\,\,\forall \xi \in \mathbb{R}^n.
\end{equation}
Set $a(x):=(a_{ij}(x))_{1 \leq i,j \leq n}$ for $x \in D$, and consider the symmetric bilinear form
$$\mathcal{E}(\varphi, \psi):= \frac{1}{2}\int_D \Big(\langle a(x)\nabla\varphi(x), \nabla\psi(x)\rangle + f(x)\varphi(x)\psi(x)\Big) \d x, \quad \varphi,\psi \in C_0^1(D),$$
where $\langle \cdot,\cdot \rangle$ denotes the inner product in $\mathbb{R}^n$ and $C_0^1(D)$ the set of all differentiable functions from $D$ into $\mathbb{R}$ with compact support. Let ${\mathcal{D}}(\mathcal{E})$ be the abstract completion of $C_0^1(D)$ {in $L^2(D)$} with respect to the norm $|f|:=\mathcal{E}(f,f)^{\frac{1}{2}}$. It turns out that {$\mathcal{D}(\mathcal{E})= W^{1,2}_0(D) \subset L^2(D)$}, the latter being the classical Sobolev space of order $1$ in $L^2(D)$ with Dirichlet boundary conditions, {and $\mathcal{E}$ can be extended by continuity as a bilinear form on it}.
Then there exists a unique self-adjoint operator $\mathcal{A}: \mathcal{D}(\A) \subset L^2(D) \to L^2(D)$ such that
$$\mathcal{E}(\varphi, \psi) = - \int_D (\mathcal{A}\varphi)(x)\, \psi(x)\, \d x, \quad \forall \varphi\in\mathcal{D}(A), \  \psi \in W^{1,2}_0(D),$$
where 
$$\mathcal{D}(\A):=\big\{\varphi \in W^{1,2}_0(D):\, W^{1,2}_0(D) \ni \psi \mapsto \mathcal{E}(\varphi, \psi)\,\,\text{is continuous w.r.t.\ $L^2(D)$-norm}\big\}.$$
Heuristically,
$$(\A\varphi)(x) = \frac{1}{2}\diver\big(a(x)\nabla\varphi(x)\big) + f(x)\varphi(x), \quad x \in D.$$
Furthermore, $(\A,\mathcal{D}(\A))$ generates a  positivity-preserving $C_0$-semigroup of linear operators $(e^{t\mathcal{A}})_{t\geq 0}\subseteq \mathcal{L}^+(L^2(D))$. For details, we refer to Section 2 in Chapter II of \cite{MaRockner92}, where an even more general situation is analyzed.

More generally, the above approach can be generalized to (not necessarily symmetric) bilinear forms of type
\begin{eqnarray*}
\label{eq:bilinear}
\displaystyle \mathcal{E}(\varphi, \psi) & := & \frac{1}{2}\int_D \Big[\langle a(x)\nabla\varphi(x), \nabla\psi(x)\rangle + \langle b(x), \nabla\varphi(x)\rangle \psi(x)  \nonumber \\
&& \displaystyle + \varphi(x) \langle h(x), \nabla\psi(x)\rangle + f(x)\varphi(x)\psi(x)\Big] \d x, \quad \varphi,\psi \in C_0^{1}(D), \nonumber 
\end{eqnarray*}
giving sense to the heuristic differential operator
\begin{eqnarray*}
& \displaystyle{ (\A\varphi)(x) = \frac{1}{2}\diver\Big(a(x)\nabla\varphi(x) + h(x)\varphi(x)\Big) - \langle b(x), \nabla\varphi(x)\rangle + f(x)\varphi(x), \quad x \in D,}
\end{eqnarray*}
in such a way that $(\A,\mathcal{D}(\A))$ (with $\mathcal{D}(\A)$ constructed similarly as above) generates a positivity preserving $C_0$-semigroup of linear operators $(e^{t\A})_{t\geq0} \subseteq \mathcal{L}^+(L^2(D))$. The following assumptions on the coefficients are sufficient for this:
$a_{ij} \in L^1_{\text{loc}}(D)$, $1 \leq i,j\leq d$, and \eqref{eq:unellip} holds; $f \in L^1_{\text{loc}}(D)$ and lower bounded; if $b=(b_1,\dots,b_n)$, $h=(h_1,\dots,h_n)$, each of the components $b_i$ or $h_i$ should belong to $L^1_{\text{loc}}(D)$ and decomposable as a sum of two functions $c_1 + c_2$, where $c_1 \in L^{\infty}(D)$, $c_2 \in L^{p}(D)$ for some $p\geq n$, with $n\geq 3$. For details see Theorem 2.2 in \cite{MaRo95}.

Then \eqref{eq:abstract}, with the specifications $\mathbf{y}:=y_0(\cdot)\in L^2(D)$ and $\nu_t:=c(t,\cdot)$, corresponds to the singularly controlled (random) PDE 
\begin{equation*}
%\label{eq:Dir}
\begin{cases}
\displaystyle{\frac{\partial y}{\partial t} (t,x)=\frac{1}{2} \diver\Big(a(x)\nabla y(t,x) + h(x)y(t,x)\Big) - \langle b(x), \nabla y(t,x)\rangle + f(x)y(t,x)+ \d c(t,x),\,\, t\in[0,T],}\\\\
y(0^-,x)=y_0(x), \ \ \ \  x\in D,\\\\
y(t,x)=0, \ \ \ \forall  (t,x)\in[0,T]\times \partial D.
\end{cases}
\end{equation*}

In the special case $a_{ij} = \delta_{ij}$, $1 \leq i,j\leq n$, and $f=b=h=0$, $\mathcal{A}=\Delta$, with $\mathcal{D}(\A)=W^{1,2}_0(D)\cap W^{2,2}(D)$ (where $W^{2,2}(D)$ is the Sobolev space of order $2$ in $L^2(D)$). In this case the semigroup $(e^{t\mathcal{A}})_{t\geq 0}$ is just the transition semigroup of Brownian motion $B$ on $D$ with absorbing vanishing boundary condition on $\partial D$; i.e.
$$(e^{t\mathcal{A}} \varphi)(x)=\E\left[\varphi(B_{t})\mathds{1}_{\{\tau>t\}}\right], \ \ \ x\in D,\, \varphi\in L^2(D),$$
with $\tau$ being the lifetime of $B$. 

If $D= \mathbb{R}^n$, it is just the Brownian semigroup, i.e. 
$$(e^{t\mathcal{A}} \varphi)(x)=\E\left[\varphi(B_{t})\right] = \frac{1}{(2\pi t)^{\frac{n}{2}}} \int_{\mathbb{R}^n} \varphi(y) e^{-\frac{1}{2t}|x-y|^2_{\mathbb{R}^n}} \d y, \quad  x \in \mathbb{R}^n,\,\varphi\in L^2(\R^n).$$

\subsection{Compact 1-dimensional manifold without boundary}
\label{sec:compactmani}

Let $S^1\cong \mathbb{R}/\mathbb{Z}$ and identify the functional spaces on $S^1$ with the corresponding functional spaces of $1$-periodic functions on $\R$; the derivatives of $\varphi:S^1\to\R$ are intended as the derivatives of this periodic function. Similarly to what we have done in Subsection \ref{sec:Dirichlet}, we can embed into our abstract setting the following singularly controlled (random) parabolic PDE on $S^1$:
\begin{equation}
\label{eq:S1}
\begin{cases}
\displaystyle{\frac{\partial y }{\partial t} (t,x)=\frac{1}{2}\frac{\partial }{\partial x}\left(a(x)\frac{\partial y}{\partial x}(t,x) 
\right)+f(x)y(t,x) + \d c(t,x), \ \ \ t\in[0,T],}\\\\
y(0^-,x)=y_0(x), \ \ \ \  x\in S^1.
%y(t,0)=y(t,2\pi), \ \ \ \frac{\partial}{\partial x} y(t,0)= \frac{\partial}{\partial x} y(t,2\pi), \ \ \ \ t\in[0,T].
\end{cases}
\end{equation}
In particular, this can be accomplished by taking $X:=L^2(S^1)$ and considering $(\mathcal{A},\mathcal{D}(A))$ defined by
$$\mathcal{D}(\mathcal{A})=  W^{2,2}(S^1); \ \ \ \ 
(\mathcal{A} \varphi)(x)=\frac{1}{2}\frac{\d }{\d x}\left(a(x)\frac{\d}{\d x} \varphi(x)\right)+f(x)\varphi(x), \ \ \ x\in S^1, \ \ \ \varphi\in \mathcal{D}(\mathcal{A}).$$   
Such a kind of settings have been considered in recent works on economic growth with geographical dimension in a deterministic and non-singular framework (see \cite{BCF} and \cite{BFFG}).

%%%%%%%%%%%%%%%%%%%%%%%%%%%%%%%%%%%%%%%%%%%%%%%%%%%%

\section{Concluding Remarks}
\label{sec:conclusion}

In this paper we have studied a class of infinite-dimensional singular stochastic control problems in which the controlled dynamics evolves according to an abstract evolution equation. We have completely characterized optimal controls through necessary and sufficient first-order conditions and we have determined the explicit form of the optimal control in a case study. There are several directions towards which our study can be extended and further developed, and we briefly discuss three relevant ones in the following.
\vspace{0.15cm}

\emph{Singular Control of SPDEs.}\ The singularly controlled abstract evolution equation \eqref{eq:abstract} does not contain any noisy term. We believe that our approach might be successfully employed also in the case in which the controlled state process $Y$ evolves according to the SPDE (see, e.g., \cite{DPZ} and \cite{LR})
$$\d Y_t = \mathcal{A} Y_t\, \d t + \mu(t,Y_t)\d t+\sigma(t,Y_t) \d W_t + \d \nu_t, \quad t\geq 0, \qquad Y_{0^-}=\mathbf{y} \in K_+,$$
where $W$ is a cylindrical Brownian motion, and $\mu,\sigma$ are suitable drift and diffusion coefficients. If $\mu,\sigma$ are linear maps, under the assumption of a concave payoff functional to be maximized (or convex cost functional to be minimized), the linearity of the controlled state dynamics with respect to the control process should enable the derivation of necessary and sufficient first-order conditions for optimality as those developed above in this paper. Notice that already the Ornstein-Uhlenbeck case of vanishing $\mu$ and constant $\sigma$ represents an interesting problem that we leave for future research.
\vspace{0.15cm}

\emph{Infinite-dimensional Bank-El Karoui's Representation Theorem.}\ In finite-dimensional settings, the Bank-El Karoui's representation theorem \cite{BankElKaroui} is known to be a powerful tool to tackle problems of (monotone) singular stochastic control and optimal stopping that do not necessarily enjoy a Markovian structure (\cite{BankRiedel1}, \cite{BankFollmer}, \cite{Bank05}, and \cite{CFR}, among others). In the case study of Section \ref{sec:explicit}, it has been possible to reduce the dimensionality of our problem and then to suitably employ the Bank-El Karoui's theorem so to find an explicit solution. A natural question that deserves to be investigated is whether an infinite-dimensional version of Bank-El Karoui's representation theorem can be proved. Clearly, this would require a careful separate analysis that is outside the scopes of the present work.
\vspace{0.15cm}

\emph{Series Expansion Analysis for Diagonal Operators.}\ When the operator $\mathcal{A}$ admits a spectral decomposition -- which is the case, e.g., of some of the second order differential operators considered in Section \ref{Sec:PDE} -- it would be interesting to try to exploit such a decomposition in order to reduce the abstract complexity of our first order conditions, at least for profit functions $\Pi$ with specific separable structures. This study is also left for future research.

%%%%%%%%%%%%%%%%%%%%%%%%%%%%%%%%%%%%%%%%%%%%%%%%%%%%%%%%%%%%%%%%%%%%%%%%%%%%%%%%%%%%%%%%%%%%%%%%%%%%%%%%%%%%%%%%%%%%%%%%%%%%%%%%%%%%%%%%%%%%%%%

\medskip

\indent \textbf{Acknowledgments.} Financial support by the German Research Foundation (DFG) through the Collaborative Research Centre 1283 ``Taming uncertainty and profiting from randomness and low regularity in analysis, stochastics and their applications'' is gratefully acknowledged by the authors. 

%%%%%%%%%%%%%%%%%%%%%%%%%%%%%%%%%%%%%%%%%%%%%%%%%%%%%%%%%%%%%%%%%%%%%%%%%%%%%%%%%%%%%%%%%%%%%%

\end{document}